\documentclass[12pt]{amsart}
\usepackage{amssymb,upref,enumerate}

\usepackage{amsmath,amssymb,amscd,amsthm,amsxtra}
\usepackage{latexsym}
\usepackage{graphics,epsfig}
\allowdisplaybreaks

\def\Xint#1{\mathchoice
{\XXint\displaystyle\textstyle{#1}}%
{\XXint\textstyle\scriptstyle{#1}}%
{\XXint\scriptstyle\scriptscriptstyle{#1}}%
{\XXint\scriptscriptstyle\scriptscriptstyle{#1}}%
\!\int}
\def\XXint#1#2#3{{\setbox0=\hbox{$#1{#2#3}{\int}$}
\vcenter{\hbox{$#2#3$}}\kern-.5\wd0}}

\def\dashint{\Xint-}

\newtheorem{theor}{Theorem}
\newtheorem{lemma}{Lemma}

\newtheorem{corollary}{Corollary}
\theoremstyle{remark}
\newtheorem{remark}{Remark}

\newtheorem{definition}{Definition}

\numberwithin{equation}{section}

\newcommand{\qqq}{\quad\quad\quad}

\newcommand{\abs}[1]{\left\lvert#1\right\rvert}
\newcommand{\norm}[2]{{\left\| #1 \right\|}_{#2}}
\newcommand{\bey}{\begin{eqnarray*}}
\newcommand{\eey}{\end{eqnarray*}}

\newcommand{\ra}{\rightarrow}

\newcommand{\al}{\alpha}

\newcommand{\ep}{\epsilon}

\newcommand{\vp}{\varphi}

\newcommand{\rn}{{\mathbf R}^n}

\newcommand{\rr}{\mathbf R}
\newcommand{\Z}{\mathbf Z}

\newcommand{\N}{\mathbf N}

\newcommand{\cc}{\mathbf C}

\newcommand{\cs}{\mathcal S}

\newcommand{\cm}{\mathcal M}

\newcommand{\cd}{\mathcal D}

\newcommand{\diam}{{\rm{diam}}}

\newcommand{\fty}{\vec{f}(\vec{y})}
\newcommand{\vf}{\vec{f}}

\newcommand{\vy}{\vec{y}}
\newcommand{\dy}{\,\, d\mu(\vec{y})}

\begin{document}

%\dedicatory{}



\address{Diego Maldonado, Department of Mathematics, Kansas State University, 138 Cardwell Hall,
Manhattan, KS-66506, USA.} \email{dmaldona@math.ksu.edu}

\address{Kabe Moen, Department of Mathematics,
Washington University in St. Louis, St. Louis, MO-63130-4899,
USA.}\email{moen@math.wustl.edu}

\address{Virginia Naibo, Department of Mathematics, Kansas State University, 138 Cardwell Hall,
Manhattan, KS-66506, USA.} \email{vnaibo@math.ksu.edu}

\thanks{First author partially supported by the NSF under grant DMS 0901587.}

\title[Weighted multilinear Poincar\'e inequalities]{Weighted multilinear Poincar\'e inequalities for vector fields of H\"ormander type}
\author{Diego Maldonado \and Kabe Moen \and Virginia Naibo}
\date{\today}

\begin{abstract}
As the classical $(p,q)$-Poincar\'e inequality is known to fail for $0 < p < 1$, we introduce the notion of weighted multilinear Poincar\'e inequality as a natural alternative when $m$-fold products and $1/m < p$ are considered. We prove such weighted multilinear Poincar\'e inequalities in the subelliptic context associated to vector fields of H\"ormader type. We do so by establishing multilinear representation formulas and weighted estimates for multilinear potential operators in spaces of homogeneous type.
\end{abstract}

\maketitle

\bigskip

\section{Introduction and main result}

The classical Poincar\'e inequality
\begin{equation}\label{poincare}
\left(\int_{B}\abs{u(x)-u_B}^q\,dx\right)^{1/q}\le C\,
\left(\int_{B}\abs{\nabla u(x)}^p\,dx\right)^{1/p},\qquad u\in
C^1(\overline{B}),
\end{equation}
where $B$ is an Euclidean ball in $\rn$ and
$u_B=\frac{1}{\abs{B}}\int_B u(x)\,dx,$ holds when $1\le p<n$ and $
q=\frac{np}{n-p}.$ However, simple examples prove that this inequality is
false for every $0<p<1,$ see for instance Buckley-Koskela~\cite[p.224]{BK94}, where it is shown that even the following weaker version of
\eqref{poincare} fails for $0<p<1$ and any Euclidean ball $B\subset
\rn,$
\begin{equation}\label{PoinWeak}
\inf\limits_{a \in \rr} \left( \int_B |u(x) - a|^q \, dx \right)^{1/q} \leq C \left(\int_B |\nabla u(x)|^p \, dx\right)^{1/p}.
\end{equation}
We mention in passing that \eqref{poincare} does hold for some $0<p<1 $ if  $u$ satisfies extra
conditions such as being a solution to a suitable elliptic PDE
(Haj{\l}asz-Koskela~\cite{HK00}, Chapter 13) or having $|\nabla u|$
bounded by a weight with a weak reverse H\"older inequality
(Buckley-Koskela~\cite{BK94}).

We now focus on the case $u=fg$ with $f,\,g\in C^1{(\overline{B})}.$ By the previous comment,
the following Poincar\'e inequality for the product of two functions
\begin{equation}\label{Poinpq}
\inf\limits_{a \in \rr} \left( \int_B |(fg)(x) - a|^q \, dx \right)^{1/q} \leq C \left(\int_B |\nabla(fg)(x)|^p \, dx\right)^{1/p},
\end{equation}
($C$ independent of $f$ and $g$), also fails for every $0 < p < 1$ and every Euclidean ball $B$. On the other hand, note that for any numbers $0 < p, r, s, \tilde{r}, \tilde{s} < \infty$ with
\begin{equation}\label{holderrelation}
\frac{1}{p} = \frac{1}{r} + \frac{1}{s} = \frac{1}{\tilde{r}} + \frac{1}{\tilde{s}},
\end{equation}
the inequality
\begin{align}\label{holderrs}
\left(\int_B |\nabla(fg)|^p \right)^{1/p}& \lesssim \left(\int_B |\nabla f|^r \, dx\right)^{1/r} \left(\int_B |g|^s \, dx\right)^{1/s}\\\nonumber
 & +  \left(\int_B |f|^{\tilde{r}} \, dx\right)^{1/{\tilde{r}}} \left(\int_B |\nabla g|^{\tilde{s}} \, dx\right)^{1/{\tilde{s}}},
\end{align}
with constant depending only on $p$, holds as a consequence of
H\"older's inequality. Hence,  a natural alternative to \eqref{Poinpq}
for arbitrary functions $f$ and $g$ and $0<p<1$ is given by the
inequality
\begin{align}\label{appl}
& \inf_{a\in\rr}\left( \int_B |(fg)(x) - a|^q \, dx
\right)^{1/q}\\\nonumber & \leq C \left( \left(\int_B |f|^r \,
dx\right)^{1/r} \left(\int_B |\nabla g|^s \, dx\right)^{1/s}+
\left(\int_B |\nabla f|^{\tilde{r}} \, dx\right)^{1/{\tilde{r}}}
\left(\int_B |g|^{\tilde{s}} \, dx\right)^{1/{\tilde{s}}} \right),
\end{align}
or the following stronger inequality, which we will call {\it bilinear Poincar\'e
inequality}
\begin{align}\label{appl1}
& \left( \int_B |(fg)(x) - f_Bg_B|^q \, dx \right)^{1/q}\\\nonumber
& \leq C \left( \left(\int_B |f|^r \, dx\right)^{1/r} \left(\int_B
|\nabla g|^s \, dx\right)^{1/s}+ \left(\int_B |\nabla f|^{\tilde{r}}
\, dx\right)^{1/{\tilde{r}}} \left(\int_B |g|^{\tilde{s}} \,
dx\right)^{1/{\tilde{s}}} \right),
\end{align}
where $p,\,r,\,s,\,\tilde{s},\,\tilde{r}$ are related through
\eqref{holderrs} and $C$ is independent of $f$, $g$, and $B$.\\

The purpose of this article is to derive weighted inequalities  of the type
\eqref{appl1}  where $p$ is allowed to be bigger than $1/2$ and, more generally, $p > 1/m$ when $m$ factor functions are involved. Moreover, we do so in the subelliptic setting associated to vector fields satisfying H\"ormander's condition. Since D. Jerison's fundamental work \cite{J86} on subelliptic Poincar\'e inequalities, the research on Poincar\'e-type inequalities in stratified groups and more general Carnot-Carath\'eodory structures has continued to gain substantial momentum, see, for instance, \cite{BKL95, CDG94, CDG97, DGPa, DGPb, FGW94, FLW95, FLW95b, FLW96, FW99, GN96, HK00, Lu92, Lu94, LW98a, LW98b, LW00, MSal95, PW01, SW92} and references there in. In particular, Buckley-Koskela-Lu \cite{BKL95} have established the validity of weighted versions of \eqref{poincare} for $0 < p < 1$ in the Carnot-Carath\'eodory setting under the assumption that the subelliptic gradient of $u$ satisfies a weak reverse H\"older condition. Along these lines, our main result is motivated by the exploration of inequalities such as \eqref{appl1} and the search for a substitute to \eqref{poincare} in the case $0 < p < 1$, also in the general Carnot-Carath\'eodory setting,  when the function $u$ in question is an $m$-fold product of differentiable functions, with no extra assumptions. Namely, we prove

\begin{theor}\label{poincareineqTheorem2w} Let $m \in \N,$ $\frac{1}{m}<p\le q<\infty$ and $1<p_1, \cdots,
 p_m <\infty$ such that
 $ \frac{1}{p}=\frac{1}{p_1}+\cdots+\frac{1}{p_m}.$
For a connected bounded open set $\Omega\subset \rn,$ let $Y = \{Y_k\}_{k=1}^M$ be a collection of vector fields on $\Omega$ verifying
H\"ormander's condition and denote by $\rho$ the associated
Carnot-Carath\'eodory metric. Let  $u, \, v_k,$ $k=1,\cdots, m,$  be
weights defined on an open set $\Omega_0\subset\subset\Omega$
and satisfying condition \eqref{2wq>1} if $q>1$ or condition
\eqref{2wq<1} if $q\le 1,$ where

\begin{equation}\label{2wq>1}
\sup_{B = B_\rho(x,r),\, x \in \overline\Omega_0}
\diam_\rho(B) \abs{B}^{1/q - 1/p} \left( \frac{1}{\abs{B}}\int_B
u^{qt} dx \right)^{1/q t} \prod_{j=1}^m \left(
\frac{1}{\abs{B}}\int_B v_i^{-t p_i'} dx \right)^{1/t p_i'} <
\infty,
\end{equation}
for some $t> 1,$

\medskip

\begin{equation}\label{2wq<1}
\sup_{B = B_\rho(x,r),\, x \in \overline\Omega_0}
\diam_\rho(B) \abs{B}^{1/q - 1/p} \left( \frac{1}{\abs{B}}\int_B
u^{q} dx \right)^{1/q} \prod_{j=1}^m \left( \frac{1}{\abs{B}}\int_B
v_i^{-t p_i'} dx \right)^{1/t p_i'} < \infty,
\end{equation}
for some $t > 1$, where $|B|$ denotes the Lebesgue measure of the $\rho$-ball $B$. Then, there exist positive constants $r_0$ and $C$ such that for all $\rho$-ball $B\subset\overline{\Omega_0}$ with radius less than $r_0$ and for all $f_k\in C^1(\overline{B}),$ $k=1,\cdots,m,$  the following weighted $m$-linear subelliptic Poincar\'e inequality holds true
\begin{align}\label{poincareineq}
&\left(\int_B\left(\abs{\prod_{k=1}^m f_k-\prod_{k=1}^m {f_k}_B}u\right)^q\,dx\right)^{1/q}\\&\le C\,
\sum_{k=1}^m\left(\int_B\left(\abs{Yf_k}
v_k\right)^{p_k}\,dx\right)^{1/p_k}\prod_{i\neq
k}\left(\int_B\left(\abs{f_i}
v_i\right)^{p_i}\,dx\right)^{1/p_i}, \nonumber
\end{align}
where
$$
{f_k}_B = \frac{1}{|B|} \int_B f_k(x) \, dx, \quad k = 1, \ldots, m.
$$
\end{theor}

\begin{remark} We point out that Theorem \ref{poincareineqTheorem2w}, as well as notion of weighted multilinear Poincar\'e inequality \eqref{poincareineq}, are new even in the Euclidean setting. When $m=2$, Theorem \ref{poincareineqTheorem2w} provides a substitute to \eqref{Poinpq} for $p > 1/2$, and, in general, for $p$ as close to $0$ as desired, as long as $m$ factor functions, with $m>1/p$, are considered.

\end{remark}

\begin{remark} It is clear that, when $p < 1$, inequality \eqref{appl1} cannot follow from an application of the linear Poincar\'e inequality \eqref{poincare}. In addition, \eqref{appl1}  does not seem to follow (at least in
a straightforward way) from an application of the linear Poincar\'e
inequality even in cases when $ p > 1$. Indeed, to illustrate why
the linear approach breaks down, in the Euclidean setting consider the particular choices
$n=2$, $r=\tilde{r} = 2$, $s=\tilde{s}=4$, and $p=4/3$, which yields
$q=4$. If we write
\begin{equation}
|f(x)g(x) - f_B g_B|^4 \lesssim |f(x)|^4 |g(x) - g_B|^4 + |g_B|^4 |f(x) - f_B|^4
\end{equation}
and use, for instance, H\"older's inequality with any auxiliary index $l \geq 1$ in the first summand to get
$$
\left( \int_B |f(x)|^4 |g(x) - g_B|^4 \, dx\right)^{1/4} \leq \left(\int_B |f(x)|^{4l'} \, dx\right)^{1/4l'} \left(\int_B |g(x) - g_B|^{4l}\, dx \right)^{1/4l}
$$
we realize that it is impossible to utilize a linear Poincar\'e inequality of the type
\begin{equation}\label{impossiblePoin}
\left(\int_B |g(x) - g_B|^{4l} \, dx\right)^{1/4l} \leq C \left(\int_B |\nabla g|^{s} \, dx\right)^{1/s},
\end{equation}
since
$$
\frac{1}{s}- \frac{1}{n} = \frac{1}{4}- \frac{1}{2} = -\frac{1}{4} \ne \frac{1}{4l}
$$
for any $l \geq 1$. We then notice that any attempt to use a linear
Poincar\'e inequality with these exponents $s$ and $r$ will be
unsuccessful, since in this example we have $1/s - 1/n < 0$ and $1/r
- 1/n = 0$. As opposed to separately considering the fractions $1/r$
and $1/s$, the bilinear approach is based on the sum $1/r + 1/s$,
which verifies
\[
\frac{1}{r} + \frac{1}{s} - \frac{1}{n} = \frac{1}{4} = \frac{1}{q}.
\]
Also, if we try a different way and write
\[
|f(x)g(x) - f_B g_B|^4 \lesssim |f(x)g(x) - (fg)_B|^4 + |(fg)_B - f_B g_B|^4,
\]
then the linear Poincar\'e inequality allows to control the first summand by
\[
\left( \int_B |f(x)g(x) - (fg)_B|^4 \, dx \right)^{1/4} \leq C \left( \int_B |\nabla (fg)(x)|^{4/3} \, dx \right)^{3/4},
\]
which, in turn, can be bounded as in \eqref{holderrs}. However, given any $l \geq 1$, for the constant term $|(fg)_B - f_B g_B|$ we have,
\begin{align*}
|(fg)_B - f_B g_B|&= \frac{1}{|B|} \left|\int_B f(x) (g(x)-g_B) \, dx\right| \leq \frac{1}{|B|} \int_B |f(x)| |g(x)-g_B| \, dx\\
& \leq   \left(\frac{1}{|B|} \int_B |f(x)|^{l'} \, dx\right)^{1/l'} \left(\frac{1}{|B|} \int_B |g(x)-g_B |^{l} \, dx\right)^{1/l},
\end{align*}
so that
\begin{align*}
\int_B |(fg)_B - f_B g_B|^4 \, dx& \leq |B| \left(\frac{1}{|B|} \int_B |f(x)|^{4l'} \, dx\right)^{1/{l'}} \left(\frac{1}{|B|} \int_B |g(x)-g_B |^{4l} \, dx\right)^{1/l}\\
&= \left(\int_B |f(x)|^{4l'} \, dx\right)^{1/{l'}} \left(\int_B |g(x)-g_B |^{4l} \, dx\right)^{1/l},
\end{align*}
where we used Jensen's inequality to avoid loose powers of $|B|$. We
now see that if we intend to bound the last term by means of the
linear Poincar\'e inequality, we run into the same problem as in
\eqref{impossiblePoin} since $1/s - 1/n < 0 \neq 1/4l$.

In a sense, controlling the oscillation $|f(x)g(x) - f_B g_B|$  (rather than the oscillation $|f(x)g(x) - (fg)_B|$) requires bilinear
methods, even for some $p$ larger than $1$.
\end{remark}

\begin{remark}  In the linear case $(m=1)$, representation formulas and Poincar\'e inequalities imply embedding theorems on Campanato-Morrey spaces, see, for instance, Lu ~\cite{Lu95, Lu98} for such embeddings in the Carnot-Carath\'eodory context. In order to illustrate the multilinear analogs of these embeddings associated to Theorem \ref{poincareineqTheorem2w}, let us focus on the Euclidean setting and the bilinear case $m=2$.  Let $w \geq 0$ be a weight and for  $p, \lambda > 0$ and $f \in L^1_{loc}(\rn, w^p)$, $f$ is said to belong to the weighted
Morrey space $L^{p,\lambda}(w)$ if
$$
\norm{f}{L^{p,\lambda}(w)} = \sup_{B} \left(\frac{1}{|B|^{\lambda/n}} \int_B |f(x) w(x)|^p\, dx \right)^{1/p} < \infty,
$$
and $f$ is said to belong to the weighted Campanato space $\mathcal{L}^{p,\lambda}(w)$ if
$$
\norm{f}{\mathcal{L}^{p,\lambda}(w)} = \sup_{B} \inf_{a \in \cc} \left( \frac{1}{|B|^{\lambda/n}}\int_B \left(|f(x) - a| w(x)\right)^p \, dx\right)^{1/p}  < \infty.
$$
Then, Theorem \ref{poincareineqTheorem2w} (with $m=2$ and in the Euclidean setting) implies a variety of weighted inequalities of the form
\begin{equation}\label{KPi}
\norm{fg}{\mathcal{L}^{p,\lambda}(w)} \lesssim \norm{\nabla f}{L^{p_1,\lambda_1}(u)} \norm{g}{L^{p_2,\lambda_2}(v)} +  \norm{ f}{L^{p_1,\lambda_1}(u)} \norm{\nabla g}{L^{p_2,\lambda_2}(v)},
\end{equation}
for a larger class of weights $u, v, w$ (and, therefore, a larger range of indices $p, \lambda, p_1, \lambda_1$, $p_2$, and $\lambda_2$) than one could possibly obtain by iteration of the linear weighted estimates and H\"older's inequality. See remark \ref{rmk:moreweights}.

Inequalities of the form \eqref{KPi} are related to the so-called Kato-Ponce inequality, where the $L^p$-norm of the derivative of the product is being replaced by another measure of the oscillation (i.e., the Campanato norm) of the product, and the Morrey spaces play the role of the Lebesgue spaces.
\end{remark}

Regarding the organization of the article, we prove Theorem
\ref{poincareineqTheorem2w} in \S~\ref{secc:proofmain} after conveniently adapting the usual approach to the classical Poincar\'e inequality~\eqref{poincare}. That is, by proving a multilinear analog to the representation formula
\[
\abs{f(x)-f_B}\lesssim I_{B,1}(\abs{\nabla f})(x), \qquad x\in B,
\]
where $I_{B,1}(h)(x)=\int_B h(y)\abs{x-y}^{1-n}\,dy$ (see Corollary \ref{multrepformula} in \S~\ref{secc:proofmain}). Then, in \S~\ref{potop} we use the framework of spaces of homogeneous type to introduce a class of multilinear potential operators that includes the multilinear counterpart to $I_{B,1}$ and we establish their weighted Lebesgue estimates in \S~\ref{secc:proofialpha}. These weighted estimates are further conveyed into the context of Orlicz spaces in \S~\ref{secc:multiOrl}, producing natural multilinear alternatives to their linear counterparts and allowing for a strictly wider range of indices, see Theorem \ref{orlicz2w} and Remark \ref{rmk:orlp<1}.

\medskip

\noindent{\bf Acknowledgments.} The authors would like to thank Carlos P\'erez and Rodolfo Torres for useful conversations regarding the topics in this article. Part of the work presented here originated in the interaction among the authors
that took place during the Prairie Analysis Seminar 2008 in Lawrence, Kansas. The participation of the authors in such conference was supported in part by the National Science Foundation grant DMS 0848357.

\section{Multilinear potential operators in spaces of homogeneous
type}\label{potop}

We introduce in this section the theory of multilinear potential operators in the ample context of spaces of homogeneous type and
state their weighted boundedness properties.

Recall that a space of homogeneous type (in the sense of Coifman-Weiss \cite{CW71}) is a triple $(X,\rho,\mu),$
where $X$ is a nonempty set, $\rho$ is a quasi-metric defined on
$X,$ that satisfies
\begin{equation}\label{quasi}
\rho(x,y)\leq \kappa(\rho(x,z)+\rho(z,y))\qquad x,y,z\in X
\end{equation}
for some $\kappa\geq 1$ and $\mu$ is a Borel measure on $X$ (with respect to the
topology defined by $\rho$) such that there exists a constant $L\ge
0$ verifying
\begin{equation}\label{doublingMu}
0<\mu(B_\rho(x,2r)\le L\,\mu(B_\rho(x,r))<\infty
\end{equation}
for all $x\in X$ and $0<r<\infty,$ and where $B_\rho(x,r)=\{x\in
X:\rho(x,y)<r\}$ is the $\rho$-ball of center $x$ and radius $r.$
It can be assumed without loss of generality that the
$\rho$-balls are open subsets of $X$, see \cite{MS}. Given a ball $B=B_\rho(x,r)$ and $\theta > 0$ we will usually write $r(B)$ to denote the radius $r$ and $\theta B$ to denote $B_\rho(x, \theta r)$.

Condition \eqref{doublingMu} is known as the {\it doubling property}
of $\mu.$
We will also  assume that  $\mu$  satisfies the
{\it reverse doubling property}, that is, for all $\eta>1$ there are constants  $c(\eta)>0$ and $\delta>0$ such that
\begin{equation}\label{reversedoubling}
\frac{\mu(B_\rho(x_1,r_1))}{\mu(B_\rho(x_2,r_2))}\ge c(\eta) \,\left(\frac{r_1}{r_2}\right)^\delta,
\end{equation}
whenever $B_\rho(x_2,r_2)\subset B_\rho(x_1,r_1),$ $x_1,\,x_2\in X$ and $0<r_1,\,r_2\le\eta\, \diam_\rho(X),$ where $\diam_\rho(X)=\sup\{\rho(x,y): x,y\in X\}.$
Note that $\eta$ is not needed when $\diam_\rho(X)=\infty$ and that when $\diam_\rho(X)<\infty$ the inequality  \eqref{reversedoubling} for, say, $\eta=2$ implies $\eqref{reversedoubling}$ for any $\eta>1$ with the same value of $\delta.$

For $x, y_1, \ldots, y_m \in X$ and   $\mu$-measurable functions
$f_1, \ldots, f_m$ defined on $X$, we will write $\vec{y}=(y_1,
\ldots, y_m) \in X^m,$ $ d\mu(\vec{y}) = d\mu(y_1) \ldots
d\mu(y_m),$ $ \vec{f}=(f_1, \ldots, f_m),$ $ \vec{f}(\vec{y})=
f_1(y_1)\ldots f_m(y_m),$ and
$\rho(x,\vec{y})=\rho(x,y_1)+\cdots,+\rho(x,y_m).$  With some abuse of notation we will write $\rho(\vec{x},\vec{y})=\rho(x_1,y_1)+\cdots+\rho(x_m,y_m)$ for $\vec{x},\vec{y}\in X^m$.  Given a measurable function $g$ on $X$, we denote the average of $g$ over a measurable subset $E\subset X$  by
$$\dashint_E g \ d\mu = \frac{1}{\mu(E)}\int_E g\ d\mu.$$

For $\alpha>0$
we define the multilinear fractional integral  operator of order
$\alpha$ as
\begin{equation}\label{potopdefinition}
\mathcal{I}_{X,\alpha} (\vec{f})(x) = \int_{X^m}
\vec{f}(\vec{y})\frac{(\rho(x,\vec{y}))^\alpha}{(\mu(B_\rho(x,
\rho(x,\vec{y}))))^m}\,d\mu(\vec{y}).
\end{equation}
More generally, we define multilinear potential operators associated to a nonnegative kernel $K(x,\vy)$ as
\begin{equation} \label{generalpot}\mathcal{T}(\vf)(x) = \int_{X^m} \vec{f}(\vy) K(x,\vy) \ d\mu(\vec{y}).\end{equation}
We will always assume that the kernel $K$ is the restriction of a nonnegative continuous kernel $\tilde{K}(\vec{x},\vec{y})$ (i.e. $K(x,\vec{y})=\tilde{K}((x,\ldots,x),\vy)$ for $(x,\vy)\in X^{m+1}$) that satisfies the following growth conditions: for every $c>1$ there exists $C>1$ such that
\begin{eqnarray}\label{growth}
\tilde{K}(\vec{x},\vy)&\leq& C \tilde{K}(\vec{z},\vy) \quad {\rm if} \ \rho(\vec{z},\vec{y}) \leq c \rho(\vec{x},\vec{y}), \text{ and} \nonumber \\ \\
\tilde{K}(\vec{x},\vy)&\leq& C \tilde{K}(\vy,\vec{z}) \quad {\rm if} \ \rho(\vec{y},\vec{z}) \leq c \rho(\vec{x},\vec{y}). \nonumber
\end{eqnarray}
The reverse doubling property implies that if the growth condition \eqref{growth} is true for some $c>1,$ then it also holds for all $c>1$ with a possibly different value of $C.$

We notice that the kernel
\begin{equation}\label{potkernel}K_\alpha(x,\vy)=\frac{\rho(x,\vy)^\al}{\mu((B_\rho(x,\rho(x,\vy)))^m}\end{equation}
associated to the operator \eqref{potopdefinition} is the restriction of
\begin{equation}\label{kerneltildam}
\tilde{K}_\alpha(\vec{x},\vy)= \frac{\rho(\vec{x},\vy)^\al}{\mu(B_\rho(x_1,\rho(\vec{x},\vy)))\cdots \mu(B_\rho(x_m,\rho(\vec{x},\vy)))}.
\end{equation}

Following \cite{SW92} we define the functional $\vp$ associated to $K$ which acts on balls by
$$\vp(B)=\sup\{ K(x,\vy) : (x,\vy) \in B^{m+1}, \rho(x,\vy)\geq c\,r(B)\}$$
for a sufficiently small positive constant $c$ and for $B$ such that $r(B)\le \eta \,\diam_\rho(X),$ for some fixed $\eta>1.$
We note that the reverse doubling property \eqref{reversedoubling}  ensures that  the set $\{ K(x,\vy) : (x,\vy) \in B^{m+1}, \rho(x,\vy)\ge cr(B)\}$ is non-empty if $c$ is sufficiently small (any $c$ satisfying $0<c^\delta<c(\eta)$ will work).

Under the assumptions \eqref{growth} on $K$, we have the following properties of $\vp$.
\begin{enumerate}[(P1)]
\item\label{prop1} If  $\theta\ge 1$ and $B$ is a $\rho$-ball in $X$ with $\theta r(B)\le \eta\,\diam_\rho(X)$, and $(x,\vy)\in (\theta B)^{m+1}$ then $\vp(B)\leq C_\theta K(x,y)$ and therefore
\begin{equation}\label{doubling}\vp(B) \leq C\vp(\theta B).\end{equation}
\item\label{prop2}  If $B'\subset B$ are $\rho$-balls in $X$ with $r(B'),\,r(B)\le \eta\,\diam_\rho(X)$, then
\begin{equation} \label{revdoubling} \vp(B)\leq C\vp(B').\end{equation}
\end{enumerate}

Note that (P1) implies that $\vp(B)<\infty$. Moreover, \eqref{doubling} and \eqref{revdoubling} assure that $\varphi$ is
 well-defined in the sense that if $B_\rho(x_1,r_1)=B_\rho(x_2,r_2),$ $0<r_1,\,r_2\le \eta\,\diam_\rho(X),$ then $\varphi(B_\rho(x_1,r_1))\approx \varphi(B_\rho(x_2,r_2)).$
 We provide a short proof of property (P1) above as the proof of (P2) is similar.
Suppose $(x,\vy)\in (\theta B)^{m+1}$ and $(s,\vec{t}) \in B^{m+1}$ with $\rho(s,\vec{t})\geq cr(B)$.  If $\rho(s,\vy) \geq \rho(\vec{t},\vy)$, then
$$\rho(x,\vy)\leq 2 m \kappa \theta r(B) \leq 4 m \kappa^2\theta c^{-1} \rho(s,\vy)$$
so that $K(s,\vy) \leq CK(x,\vy)$.  Further,
$$\rho(s,\vy)\leq 2 m \kappa \theta r(B) \leq 2 m \kappa \theta c^{-1} \rho(s,\vec{t})$$
which implies $K(s,\vec{t})\leq CK(s,\vy)$, and hence
$$K(s,\vec{t})\leq C K(x,\vy).$$
In the case when $\rho(s,\vy) \leq \rho(\vec{t},\vec{y})$, we have $\rho(x,\vy) \leq c\rho(\vec{t},\vec{y})$.  Hence,
$$\tilde{K}(\vec{t},\vec{y}) \leq C\tilde{K}((x,\ldots,x),\vy)=CK(x,\vy)$$
and
$\rho(\vec{t},\vec{y}) \leq 2 m \kappa \theta r(B) \leq c\rho(s,\vec{t})$
showing
$$K(s,\vec{t})\leq C \tilde{K}(\vec{t},\vec{y})\leq CK(x,\vy).$$
Taking the supremum over the proper $(s,\vec{t})$ we have
$$\vp(B) \leq CK(x,\vy).$$

When $K(x,\vy)=(|x-y_1|+\cdots+|x-y_m|)^{\al-nm}$, we have
$$\vp(B)\approx {r(B)^{\al-nm}} $$
and when $K$ is given by \eqref{potkernel} we have
$$\vp(B)\approx \frac{r(B)^\al}{\mu(B)^m},\quad r(B)\le\eta\, \diam_\rho(X),$$
with constants that depend only on $\kappa,$ $L,$  and $c$ as in the definition of $\vp$ (and therefore on $c(\eta)$ and $\delta$).

Finally, we will assume that the functional $\vp$ associated to our kernel $K$ satisfies the following property: there exists $\ep>0$  such that for all $C_1>1$  there exists $C_2>0$ such that
\begin{equation} \label{mainassump}
\vp(B')\mu(B')^m \leq C_2\,\left(\frac{r(B')}{r(B)}\right)^\ep \vp(B)\mu(B)^m
\end{equation}
  for  all balls  $B'\subset B,$ with  $r(B'),\,r(B)< C_1\,  \diam_\rho(X)$.
 Note that the last condition is superfluous when $\diam_\rho(X)=\infty,$ and that if $\diam_\rho(X)<\infty,$ it is enough to check \eqref{mainassump} for only, say, $C_1=2,$ and that $\epsilon$ can be taken to be independent of $C_1.$

\begin{remark}\label{mainassumpremark}
Notice that $$K(x,\vy)=(|x-y_1|+\cdots +|x-y_m|)^{\al-nm}$$ and
$$K_\alpha(x,\vy)=\frac{\rho(x,\vy)^\al}{\mu(B_\rho(x,\rho(x,\vy)))^m}$$ both satisfy \eqref{mainassump} with $\ep=\al.$  In the general case $K_\alpha,$ if the constant $C_1$ depends only on   $\kappa,$ $L,$  and the constants $c(\eta)$ and $\delta$ in \eqref{reversedoubling} with $\eta=C_1,$ so does the corresponding constant $C_2.$
\end{remark}

We now state our main results concerning weighted boundedness
properties for $\mathcal{T}$.

\begin{theor}\label{general2w} Suppose that $1<p_1, \cdots,
 p_m<\infty,$ $\frac{1}{p}=\frac{1}{p_1}+\cdots+\frac{1}{p_m}$ and $\frac{1}{m}<p\le q<\infty.$
  Let $(X,\rho,\mu)$ be a space of homogeneous type  that satisfies the reverse doubling property \eqref{reversedoubling} and let  $K$ be a kernel such that \eqref{growth} holds with $\vp$ satisfying \eqref{mainassump}.  Furthermore, let $u, v_k,$
$k=1,\cdots,m$ be weights defined on $X$ that satisfy condition
\eqref{general2wq>1hm} if $q>1$ or condition \eqref{general2wq<1hm} if $q\le 1,$
where
\begin{equation}\label{general2wq>1hm}
\sup_{B\, \rho\text{-ball}} \vp(B) \mu(B)^{\frac{1}{q} + \frac{1}{{p_1}'}+\cdots +\frac{1}{{p_m'}}}
\left( \frac{1}{\mu(B)}\int_B u^{qt} d\mu \right)^{1/qt}
\prod_{j=1}^m \left( \frac{1}{\mu(B)}\int_B v_i^{-tp_i'} d\mu
\right)^{1/tp_i'} < \infty,
\end{equation}
for some $t > 1$,

\begin{equation}\label{general2wq<1hm}
\sup_{B\, \rho\text{-ball}} \vp(B) \mu(B)^{\frac{1}{q} + \frac{1}{{p_1}'}+\cdots +\frac{1}{{p_m'}}}
\left( \frac{1}{\mu(B)}\int_B u^{q} d\mu \right)^{1/q} \prod_{j=1}^m
\left( \frac{1}{\mu(B)}\int_B v_i^{-tp_i'} d\mu \right)^{1/tp_i'} <
\infty,
\end{equation}
for some $t > 1$. Then there exists a constant $C$  such that
\[
\left(\int_{X}\left(\abs{\mathcal{T}\vec{f}}u\right)^q\,d\mu\right)^{1/q}\le
C \prod_{k=1}^m\left(\int_X(\abs{f_k}v_k)^{p_k}\,d\mu\right)^{1/p_k}
\]
for all $\vec{f}\in L^{p_1}(X,v_1^{p_1} d\mu)\times\cdots\times
L^{p_m}(X,v_m^{p_m}d\mu).$ The constant $C$ depends only on the constants appearing in \eqref{quasi}, \eqref{doublingMu}, \eqref{reversedoubling}, \eqref{growth}, \eqref{mainassump}, \eqref{general2wq>1hm} and \eqref{general2wq<1hm}.
\end{theor}

When $K$ is given by \eqref{potkernel} as noted before we have $\vp(B)\approx {\diam_\rho(B)^\alpha}/{\mu(B)^m}$,  hence we have the following result for $\mathcal{I}_{X,\al}$

\begin{corollary}\label{potopbound2w} Suppose that $1<p_1, \cdots,
 p_m<\infty,$ $\frac{1}{p}=\frac{1}{p_1}+\cdots+\frac{1}{p_m}$ and $\frac{1}{m}<p\le q<\infty.$
Let $(X,\rho,\mu)$ be a space of homogeneous type that satisfies the reverse doubling property \eqref{reversedoubling} and assume that the kernel $\tilde{K}_\alpha$ in \eqref{kerneltildam} satisfies the growth conditions \eqref{growth}. Let $u, v_k,$ $k=1,\cdots,m$ be weights defined on $X$ that satisfy condition
\eqref{2wq>1hm} if $q>1$ or condition \eqref{2wq<1hm} if $q\le 1,$
where
\begin{equation}\label{2wq>1hm}
\sup_{B\, \rho\text{-ball}} \diam_\rho(B)^\alpha \mu(B)^{1/q - 1/p}
\left( \frac{1}{\mu(B)}\int_B u^{qt} d\mu \right)^{1/qt}
\prod_{j=1}^m \left( \frac{1}{\mu(B)}\int_B v_i^{-tp_i'} d\mu
\right)^{1/tp_i'} < \infty,
\end{equation}
for some $t > 1$,

\begin{equation}\label{2wq<1hm}
\sup_{B\, \rho\text{-ball}} \diam_\rho(B)^\alpha \mu(B)^{1/q - 1/p}
\left( \frac{1}{\mu(B)}\int_B u^{q} d\mu \right)^{1/q} \prod_{j=1}^m
\left( \frac{1}{\mu(B)}\int_B v_i^{-tp_i'} d\mu \right)^{1/tp_i'} <
\infty,
\end{equation}
for some $t > 1$. Then there exists a constant $C$  such that
\[
\left(\int_{X}\left(\abs{\mathcal{I}_{X,\alpha}\vec{f}}u\right)^q\,d\mu\right)^{1/q}\le
C \prod_{k=1}^m\left(\int_X(\abs{f_k}v_k)^{p_k}\,d\mu\right)^{1/p_k}
\]
for all $\vec{f}\in L^{p_1}(X,v_1^{p_1}d\mu)\times\cdots\times
L^{p_m}(X,v_m^{p_m}d\mu).$ The constant $C$ depends only on the constants appearing in \eqref{quasi}, \eqref{doublingMu}, \eqref{reversedoubling}, \eqref{growth}, \eqref{mainassump}, \eqref{2wq>1hm} and \eqref{2wq<1hm}
\end{corollary}

\begin{remark}\label{rmk:moreweights}
Moen~\cite{Mthesis, M09} proved Corollary~\ref{potopbound2w} in
the context  of $X=\rn$ with the Euclidean  metric and Lebesgue measure. The multilinear fractional integral operator
$\mathcal{I}_{X,\alpha}$ reduces to the Riesz potential of order
$\alpha$ in $\rn$ when $m=1,$ $X=\rn$ and $\mu$ is Lebesgue
measure. Namely,
\[
I_\alpha f(x)=\int_{\rn}\frac{f(y)}{\abs{x-y}^{n-\alpha}}\,dy,
\qquad 0<\alpha<n.
\]
Muckenhoupt and Wheeden~\cite{MW} characterized the one-weight
strong type inequality
\[
\left(\int_{\rn} (I_\alpha f w)^q\,dx\right)^{1/q}\le C\,
\left(\int_{\rn} (f w)^p\,dx\right)^{1/p},
\]
for $f\ge 0,$ $1<p<\frac{n}{\alpha}$ and $q$ such that
$\frac{1}{q}=\frac{1}{p}-\frac{\alpha}{n}.$ They proved that the
above inequality holds if and only if $w$ belongs to the class
$A_{p,q},$ this is
\[
\sup_{Q}\left(\frac{1}{\abs{Q}}\int_Q w^q\,dx\right)^{1/q}
\left(\frac{1}{\abs{Q}}\int_{Q}w^{-p'}\,dx\right)^{1/p'}<\infty,
\]
where the supremum is taken over all cubes $Q$ in $\rn$ with sides
parallel to the coordinate axes. The two-weight strong type
inequality for $I_\alpha$
\[
\left(\int_{\rn} (I_\alpha f w)^q\,dx\right)^{1/q}\le C\,
\left(\int_{\rn} (f v)^p\,dx\right)^{1/p}, \qquad f\ge0,
\]
was also extensively studied. For
example, Sawyer~\cite{S88} gave a characterization for $w$ and $v$
that basically come to testing the above inequality with $f=\chi_Q
v^{(1-p')p}$ and its dual inequality with $\chi_Q w^q.$
Sawyer-Wheeden~\cite{SW92} and P\'erez-Wheeden~\cite{PW01} studied
two-weight conditions for weighted inequalities of
fractional integral operators on spaces of homogeneous type.
In particular, \eqref{2wq>1hm} ($q>1$)  reduces to the conditions imposed in  P\'erez-Wheeden~\cite[Theorem 2.2]{PW01}
to prove weighted boundedness properties  for $\mathcal{I}_{X,1}$ when $m=1.$

We stress that, even in the Euclidean setting and with the choice $u = \prod_{k=1}^m v_i$, using iterations of the linear results mentioned above to prove multilinear ones would lead to considering weights $(v_1, \ldots, v_m)$ in the class
$$
W(p,q) :=\bigcup_{q_1, \ldots, q_m} \prod_{i=1}^m A_{p_i,q_i},
$$
where the union is over all $q_i \geq p_i$ that satisfy $1/q = 1/q_1 + \cdots + 1/q_m$, $1/p_i - 1/q_i = 1/n$, $i=1, \ldots, m$. However, the class of weights $u$, $v_1$,$\ldots$, $v_m$ (with $u = \prod_{k=1}^m v_i$) satisfying \eqref{2wq>1hm} is strictly larger than $W(p,q)$. See Remark 7.5 in \cite{M09} and Section 7 in \cite{LOPTT}. See also Pradolini \cite{P10} for related results on weighted inequalities for
the multilinear fractional integral operator on $\rn$.

\end{remark}

\section{Proof of Theorem~\ref{general2w}}\label{secc:proofialpha}

We recall the following construction due to M. Christ \cite{Ch90} of dyadic cubes in a general space of homogeneous type $(X,\rho,\mu)$ with constant
$\kappa\ge 1$ in the quasi-triangle inequality for $\rho$. There exists a collection of open subsets
$\mathcal{D}= \{Q_\alpha^k : k \in \Z, \alpha \in I_k\}$ (here, for each
$k \in \Z$, $I_k$ is a countable set of indices), and constants $A >
2\kappa$, $a_0, a_1 > 0$, depending only on $\kappa,$ such that
\begin{enumerate}[(i)]
\item $\mu\left(X \setminus \bigcup_\alpha Q_\alpha^k \right)=0$
for every $k \in \Z$,

\item\label{nested} given $Q_\beta^l$ and $Q_\alpha^k$ with $l \leq k$,
then either $Q_\beta^l \subset Q_\alpha^k$ or
$Q_\beta^l \cap Q_\alpha^k = \emptyset$,

\item\label{uniqueparent} for each $(k,\alpha)$ and each $l > k$ there is
a unique
$\beta$ such that $Q_\alpha^k \subset Q_\beta^l$,

\item\label{diamlak} $\diam_\rho(Q_\alpha^k) \leq a_1 A^k$,

\item\label{ballinside} each $Q_\alpha^k$ contains some ball
$B_\rho(x_\alpha^k,a_0
A^k)$,

\end{enumerate}

We set $\mathcal{D}^k=\{Q_\al^l \in \mathcal{D}: l=k\}$ and note that by
property \eqref{nested} the family $\mathcal{D}^k$ may be assumed to be
disjoint.    If $Q=Q_{\alpha}^k$ we call
$x_Q=x_\alpha^k$ as given in property \eqref{ballinside} the center of $Q$
and define $B(Q)=B_{\rho}(x_Q,2\kappa a_1 A^{k})$ where $a_1$ is as given
in property \eqref{diamlak}. Note that if $\diam_\rho(X)=\infty$ then $Q\neq X$ for all $Q\in \mathcal{D},$ and if $\diam_\rho(X)<\infty,$ there exists $k_0\in\Z$
such that $\mathcal{D}^k=\{X\}$ for all $k\ge k_0$ and $X$ is not in $\mathcal{D}^k$ for $k\le k_0,$ in which case we only consider $k\le k_0.$

Observe that if $Q\in \mathcal{D}^k,$ $Q'\in \mathcal{D}^{k'}$ and
$Q\subset Q'$ then
$B(Q)\subset B(Q').$ To see this, note that
by property \eqref{nested}, we have $k\le k'.$ Then if $y\in B(Q)$
\[
\rho(y,x_{Q'})\le \kappa \,(\rho(y,x_Q)+\rho(x_Q,x_{Q'}))\le
\kappa\,(2\kappa a_1A^{k}+a_1 A^{k'})\le 2\kappa\,a_1 A^{k'},
\]
where we have used property \eqref{diamlak} for the cube $Q'$ and that
$A>2\kappa.$

Notice that by property \eqref{uniqueparent}  for every $Q\in \cd^k$ there
is a unique cube $Q^*\in \cd^{k+1},$ called the parent of $Q,$ such that
$Q\subset Q^*.$ Moreover $\mu(Q)\sim\mu(B(Q))\sim \mu(B(Q^*))$   since
$\mu$ is doubling and
\[B_\rho(x_Q,a_0 A^k)\subset Q\subset B(Q)\subset B(Q^*)\subset
(\kappa+1/2)A\,B(Q) .\]
It is important to observe that if $Q\in \mathcal{D}^k$ and $Q\neq X$ there exists $l>k$ such that if $Q^{*l}\in \mathcal{D}^l$ is the cube containing $Q$ given by property \eqref{uniqueparent} (the $l$th ancestor of $Q$) then $Q\subsetneq Q^{*l}.$  This is clear when $\diam_\rho(X)<\infty.$ When $\diam_\rho(X)=\infty,$ if $Q=Q^{*l}$ for all $l>k$ then, by property \eqref{ballinside}, $Q$ contains balls of radius $a_0 A^l$ for all $l>k.$
However, the radius and diameter of a ball are comparable (a consequence of the reverse doubling property \eqref{reversedoubling}), obtaining $A^l\le C\, \diam_\rho(Q)<\infty$ for all $l>k,$ a contradiction.

Our first step towards the proof of Theorem~\ref{general2w} is a
discretization of $\mathcal{T}$.  Let $(x,\vy)\in (\bigcap_k
\bigcup_\alpha Q_{\alpha}^k)^{m+1}$ and $l\in \Z$ be such that
$$A^{l-1}\leq \rho(x,\vy)\leq A^{l}.$$
There is a dyadic cube $Q\in \cd^l$ with $x\in Q$.  Let $x_Q$ be the
center of $B(Q)$, and $y_1,\ldots,y_m$ be the coordinates of $\vy$.
Since $\diam(Q)\le a_1 A^l$ (and we can assume that $a_1$ is larger than
1),
$$\rho(x_Q,y_i)\leq \kappa(\rho(x,x_Q)+\rho(x,y_i))\leq \kappa
(a_1+1)A^{l}\leq 2\kappa a_1 A^l$$
for $1\leq i\leq m$ and consequently $\vy\in B(Q)^m$.  Furthermore, since
$(x,\vy)\in B(Q)^{m+1}$ and $\rho(x,\vy)\geq A^{l-1}= r(B(Q))/2\kappa a_1 A$ we have
$$K(x,\vy)\leq \vp(B(Q))$$
by the definition of $\vp$ (note that $r(B(Q))=2\kappa a_1 A^l\le 2\kappa a_1 A\rho(x,\vec{y})\le  2 \kappa a_1A m\, \diam_\rho(X)$, so we can choose a structural constant $\eta\ge 2\kappa a_1 A m$ in the definition of $\vp$).  Since $x\in Q$ and $\vy \in B(Q)^m$ it follows that
$$K(x,\vy)\leq  \vp(B(Q))\chi_Q(x) \chi_{B(Q)^m}(\vy) \leq  \sum_{Q\in
\cd} \vp(B(Q))\chi_Q(x) \chi_{B(Q)^m}(\vy)$$
where the last inequality holds for almost all $(x,\vy)\in X^{m+1}$.  Multiplying by $\vf(\vy)\ge 0$
and integrating yields
\begin{equation} \label{discretize}\mathcal{T}(\vf)(x)\leq \sum_{Q\in \cd}
\vp(B(Q))\int_{B(Q)^m} \vf(\vy)\  d\mu(\vy) \chi_Q(x).\end{equation}
Multiplying by $u(x)g(x)\geq0 $ and integrating
$$\int_X \mathcal{T}(\vf)(x) g(x) u(x) \ d\mu(x) \leq \sum_{Q\in \cd}
\vp(B(Q))\int_Q g(x)u(x) \ d\mu(x)  \int_{B(Q)^m} \vf(\vy) \ d\mu(\vy).$$

Now we switch the summation to a smaller set of dyadic cubes with better
disjointness properties.  To define this smaller set of dyadic cubes we look
at level sets corresponding to a certain multilinear maximal function.
Set
$$\cm_{B(\cd)}(\vec{h})(x)=\sup_{Q\in \cd:x\in
Q}\frac{1}{\mu(B(Q))^m}\int_{B(Q)^m} |\vec{h}(\vy)| \dy,\qquad x\in
\bigcup_{Q\in \cd} Q.$$
Let $a>1$ be a number to be chosen later, and set
$$\cs^k=\{x\in \cup_{Q\in\cd} Q: \cm_{B(\cd)}(\vec{f})(x)>a^k\}.$$
If $x \in \mathcal{S}^k$, then there exists $Q \in \cd$ such that $x \in
Q$ and
\begin{equation}\label{m>ak}
\frac{1}{\mu(B(Q))^m}\int_{B(Q)^m} \vf (\vy) \dy > a^k.
\end{equation}
In particular, we have $Q \subset \mathcal{S}^k$ and the fact that
$\int_{X^m} \fty \dy < \infty$ and the nested nature of the dyadic cubes
in $\cd$ allow to write
$$
\mathcal{S}^k = \bigcup_j Q_{k,j},
$$
where the cubes $Q_{k,j}$ belong to $\cd$, and they are disjoint and
maximal relative to inclusion and generation  with respect to the property \eqref{m>ak} (the existence of these maximal
cubes is guarantied by the reverse doubling property \eqref{reversedoubling} when $\diam_\rho(X)=\infty$).  Notice that if
$Q_{k,j}^*$ is the parent of  $Q_{k,j}$ and $Q_{k,j}\neq X,$  by the maximality of
$Q_{k,j}$  we have
\bey
a^k&<&\frac{1}{\mu(B(Q_{k,j}))^m}\int_{B(Q_{k,j})^m} \vf (\vy) \dy \\
&\leq & \frac{c}{\mu(B(Q_{k,j}^*))^m}\int_{B(Q_{k,j}^*)^m} \vf (\vy) \dy
\\
&\leq & c a^k \leq a^{k+1}
\eey
if  $a$ is chosen large enough.

The next step is to estimate $\mu(Q_{k,j} \cap \mathcal{S}^{k+1})$.
Consider $x \in Q_{k,j} \cap \mathcal{S}^{k+1},$ then
\begin{equation}
\cm_{B(\cd)}(\vec{f})(x)=\mathop{\sup_{P \in \mathcal{D}}}_{x \in P}
\frac{1}{\mu(B(P))^m}\int_{B(P)^m} \vf( \vy) \dy > a^{k+1}
\end{equation}
and the nested property of dyadic cubes together with the maximality  of
$Q_{k,j}$  with respect to the inequality \eqref{m>ak} imply that if $P \in
\cd$ is such that $x\in P$ and
$$\frac{1}{\mu(B(P))^m}\int_{B(P)^m} \vf(
\vy) \dy > a^{k+1},$$
then  $P \subset Q_{k,j}.$ Therefore, we have
\begin{align*}
a^{k+1}& < \cm_{B(\cd)}(\vec{f})(x)=\mathop{\sup_{P \in \mathcal{D}}}_{x
\in P \subset Q_{k,j}} \frac{1}{\mu(B(P))^m}\int_{B(P)^m} \vf( \vy) \dy\\
& \le \mathop{\sup_{P \in \mathcal{D}}}_{x \in P }
\frac{1}{\mu(B(P))^m}\int_{B(P)^m} (f_1 \chi_{B(Q_{k,j})},\ldots, f_m
\chi_{B(Q_{k,j})}) ( \vy) \dy,
\end{align*}
where we have used that $B(P)\subset B(Q_{k,j})$ for $P\subset Q_{k,j}.$
Consequently,
\begin{align*}
\mu(Q_{k,j} \cap \mathcal{S}^{k+1}) & = \mu (\{x \in Q_{k,j} :
\mathcal{M}_{B(\mathcal{D})}(\vec{f})(x) > a^{k+1} \} )\\
& \leq \mu (\{x \in Q_{k,j} : \mathcal{M}_\mu(f_1
\chi_{B(Q_{k,j})},\ldots, f_m \chi_{B(Q_{k,j})})(x) > a^{k+1} \} )\\
& \leq \left(\frac{\|\mathcal{M}_\mu\|}{a^{k+1}}  \int_{B(Q_{k,j})^m}
\vf(\vy) \, \dy \right)^{1/m}\\
& = \mu(B(Q_{k,j}))
\left(\frac{\|\mathcal{M}_\mu\|}{a^{k+1}\mu(B(Q_{k,j}))^m}
\int_{B(Q_{k,j})^m}  \vf(\vy) \, \dy \right)^{1/m}\\
& \leq \mu(B(Q_{k,j})) \left(\frac{c\|\mathcal{M}_\mu\|}{a} \right)^{1/m} =: \theta \mu(B(Q_{k,j}))
\leq \theta\, \mu(Q_{k,j}),
\end{align*}
where $\cm_\mu$ is the multi-sublinear maximal operator
$$\cm_\mu(\vec{h})(x) =\mathop{\sup_{x\in B}}_{B\,\rho \text{-ball}} \prod_{i=1}^m \frac{1}{\mu(B)} \int_B
|h_i(y_i)| \ d\mu(y_i)$$
and $\|\cm_\mu\|$ is the smallest constant in the weak inequality
$$\mu\{ x\in X: \cm_\mu(\vf)(x)>\lambda\}^m\leq
\frac{\|\cm_\mu\|}{\lambda}\prod_{i=1}^m \|f_i\|_{L^1(\mu)}.$$
Notice that such a constant $\|\cm_\mu\|$ exists because
$$\cm_\mu(\vf) \leq \prod_{i=1}^m M_\mu f_i$$
where $M_\mu$ is the Hardy-Littlewood maximal operator associated to the space of homogeneous type $(X,\rho,\mu)$. The constant $\theta$ can be made smaller than one by choosing $a$
sufficiently large. In particular, if we set $E_{k,j}= Q_{k,j}\setminus
\mathcal{S}^{k+1}$, we get
\begin{equation}
\mu(E_{k,j}) \geq \gamma \mu(Q_{k,j}), \quad Q_{k,j}\neq X,
\end{equation}
for some constant $\gamma \in (0,1)$ that depends only on structural constants.

Note that if $\diam_\rho(X)<\infty$ then there exists $k_1\in\Z$ such that $a^{k_1}<\frac{1}{\mu(X)}\int_X f(y)\,dy\le a^{k_1+1}.$ If $\diam_\rho(X)=\infty,$
set $k_1=-\infty.$ Next, for $k > k_1,$ $k\in \Z,$ define
\begin{eqnarray*}
\mathcal{C}^k &=& \{Q \in \mathcal{D}: a^k <
\frac{1}{\mu(B(Q))^m}\int_{B(Q)^m} \vf ( \vy) \dy \leq a^{k+1}\},
\end{eqnarray*}
and
\begin{equation*}
\mathcal{C}^{k_1}=
\begin{cases} \{Q \in \mathcal{D}:
\frac{1}{\mu(B(Q))^m}\int_{B(Q)^m} \vf ( \vy) \dy \leq a^{k_1+1}\}, & k_1\neq-\infty; \\
\varnothing, \quad & k_1=-\infty.
\end{cases}
\end{equation*}

If $k>k_1,$ we have $Q_{k,j}\in \mathcal{C}^k$ for all $j$
and if $Q\in \mathcal{C}^k,$ $k>k_1,$ then $Q$ must be contained in $Q_{k,j}$ for some
$j$.  Returning to the estimate for $\int_X (\mathcal {T}\vf) g u\ d\mu $,
we have
\begin{align*}
& \int_X (\mathcal{T}\vf) gu\ d\mu\\
&\leq \sum_{Q\in \cd}
\vp(B(Q))\int_{B(Q)^m}\fty \dy \int_Q g(x)u(x) \ d\mu(x)\\
&=\sum_{k\ge k_1} \sum_{Q\in \mathcal{C}^k} \frac{1}{\mu(B(Q))^m}\int_{B(Q)^m}\fty
\dy \ \vp(B(Q))\mu(B(Q))^m\int_Q gu \ d\mu  \\
&\leq  \sum_{k>k_1} a^{k+1} \sum_j \sum_{\stackrel{Q\in \mathcal{C}^k}{Q\subset
Q_{k,j}}}  \vp(B(Q))\mu(B(Q))^m\int_Q gu \ d\mu \\
& \qqq \qqq + \sum_{Q\in \mathcal{C}^{k_1}} a^{k_1+1} \ \vp(B(Q))\mu(B(Q))^m\int_Q gu \ d\mu.
\end{align*}
We need the following lemma.
\begin{lemma}\label{packing} If $\vp$ satisfies \eqref{mainassump} with constants  $\epsilon,$ $C_1$ and $C_2$ then
there exists a constant $C=C(C_1, C_2,\epsilon,A)$ such that for each $Q_0\in \cd$ with $r(B(Q_0))\le C_1\, \diam_\rho(X)$
$$\sum_{\stackrel{Q\in \cd}{Q\subset Q_0}}  \vp(B(Q))\mu(B(Q))^m\int_Q gu
\ d\mu \leq C\vp(B(Q_0))\mu(B(Q_0))^m\int_{Q_0} gu \ d\mu.
$$
\end{lemma}
\begin{proof}
Note that if $Q\in\mathcal{D}$ and $Q\subset Q_0,$ then $r(B(Q))\le r(B(Q_0))\le C_1\,\diam_\rho(X)$ and recall that
 $B(Q)\subset B(Q_0).$ We now use the condition \eqref{mainassump} on $\vp$ to get
\begin{align*}
& \sum_{\stackrel{Q\in \cd}{Q\subset Q_0}} \vp(B(Q))\mu(B(Q))^m\int_Q gu \, d\mu\\
&= \sum_{l=0}^\infty \sum_{\stackrel{Q\subset Q_0}{\ell(Q)=A^{-l}\ell(Q_0)}}
\vp(B(Q))\mu(B(Q))^m\int_Q gu \, d\mu \\
&\leq  C_2\,\vp(B(Q_0))\mu(B(Q_0))^m \sum_{l=0}^\infty A^{-l\ep}
\sum_{\stackrel{Q\subset Q_0}{\ell(Q)=A^{-l}\ell(Q_0)}}\int_Q gu \, d\mu\\
&\leq C_2\,\left(\sum_{l=0}^\infty A^{-l\ep}\right) \vp(B(Q_0))\mu(B(Q_0))^m
\int_{Q_0} gu \, d\mu.
\end{align*}
\end{proof}

 Lemma \ref{packing} with $C_1=\eta,$ where $\eta=2\kappa a_1 A m$ is the structural constant chosen above, yields
\bey \int_X (\mathcal{T}\vf) g u\ d\mu &\leq&  C\sum_{k,j,\, k>k_1}
\vp(B(Q_{k,j}))\mu(B(Q_{k,j}))^m \prod_{i=1}^m
\dashint_{B(Q_{k,j})}f_i(y_i) \ d\mu(y_i)  \\
&&\qquad  \quad \times \dashint_{Q_{k,j}} gu \ d\mu \ \ \mu(Q_{k,j}) + C_{k_1},
\eey
where $C_{k_1}=0$ if $k_1=-\infty$ and  \[C_{k_1}=\vp(X)\mu(X)^m \prod_{i=1}^m
\dashint_{X}f_i(y_i) d\mu   \dashint_{X} gu \ d\mu\ \mu(X)\] if $k_1\neq-\infty.$
We have thus fully discretized $\mathcal{T}$ and are ready to put
everything together to get the estimates for the case $q\ge 1$. If ${k_1}\neq-\infty,$  $C_{k_1}$ can be handled in the same way as the terms in the sum on $k$ and $j,$ so we will assume that $k_1=-\infty$ and therefore $C_{k_1}=0.$
Let $[u,\vec{v}]$ represent the finite quantity in the weight condition  \eqref{general2wq>1hm}.
Using H\"older inequality and \eqref{general2wq>1hm} we have
\bey
\lefteqn{\int_X (\mathcal{T}\vf) g u\ d\mu }\\
&\leq & C\sum_{k,j} \vp(B(Q_{k,j}))\mu(B(Q_{k,j}))^m
\prod_{i=1}^m\left(\dashint_{B(Q_{k,j})} v_i^{-tp_i'} \ d\mu
\right)^{1/(tp_i')} \left(\dashint_{Q_{k,j}} u^{tq} \ d\mu\right)^{1/(tq)}
\\
&&\qquad \times
\prod_{i=1}^m\left(\dashint_{B(Q_{k,j})}(f_iv_i)^{(tp_i')'} \ d\mu(y_i)
\right)^{1/(tp_i')'}
 \left(\dashint_{Q_{k,j}} g^{(tq)'} \ d\mu\right)^{1/(tq)'} \mu(Q_{k,j})
\\
 &\leq & c[u,\vec{v}]
\sum_{k,j}\prod_{i=1}^m\left(\dashint_{B(Q_{k,j})}(f_iv_i)^{(tp_i')'} \
d\mu(y_i) \right)^{\frac{1}{(tp_i')'}}
 \left(\dashint_{Q_{k,j}} g^{(qt)'} \ d\mu\right)^{\frac{1}{(tq)'}}
\mu(Q_{k,j})^{\frac{1}{q'}+\frac{1}{p}} \\
 &\leq &c[u,\vec{v}]
\left(\sum_{k,j}\prod_{i=1}^m\left(\dashint_{B(Q_{k,j})}(f_iv_i)^{(tp_i')'}
\ d\mu(y_i) \right)^{q/(tp_i')'} \mu(Q_{k,j})^{q/p}\right)^{1/q}  \\
 &&\qquad \times \left(\sum_{k,j}  \left(\dashint_{Q_{k,j}} g^{(qt)'} \
d\mu\right)^{q'/(tq)'} \mu(Q_{k,j})\right)^{1/q'} \\
 &\leq & c[u,\vec{v}]
\left(\sum_{k,j}\prod_{i=1}^m\left(\dashint_{B(Q_{k,j})}(f_iv_i)^{(tp_i')'}
\ d\mu(y_i) \right)^{p/(tp_i')'} \mu(E_{k,j})\right)^{1/p}  \\
 &&\qquad \times \left(\sum_{k,j}  \left(\dashint_{Q_{k,j}} g^{(qt)'} \
d\mu\right)^{q'/(tq)'} \mu(E_{k,j})\right)^{1/q'} \\
 &\leq & c[u,\vec{v}]\prod_{i=1}^m \left(\int_X M_{(tp_i')'}(f_iv_i)^{p_i}
\ d\mu\right)^{1/p_i} \left(\int_X M_{(tq)'}(g)^{q'} \ d\mu\right)^{1/q'},
\\
\eey
where in the last line $M_s(g)=M_\mu(|g|^s)^{1/s}$ is the $L^s(\mu)$
average maximal function.  Notice that since $t>1$ we have
$$\int_X (\mathcal{T}\vf) g u\ d\mu \leq  c[u,\vec{v}]\prod_{i=1}^m
\|f_iv_i\|_{L^{p_i}(\mu)} \|g\|_{L^{q'}(\mu)}.$$
By duality we finally obtain
$$\|u\mathcal{T}(\vf)\|_{L^q(\mu)} \leq C \prod_{i=1}^m
\|f_iv_i\|_{L^{p_i}(\mu)}.$$

Next, we address the case when $q\leq 1$.
Since $q\leq 1$, using \eqref{discretize} we have
$$\mathcal{T}\vf(x)^q \leq \sum_{Q\in \cd} \left( \vp(B(Q))\int_{B(Q)^m}\fty\ d \mu(\vy)\right)^q \chi_Q(x)$$
and hence
$$\int_X (u\mathcal{T}\vf )^q \ d\mu \leq \sum_{Q\in \cd}
\left(\vp(B(Q))\int_{B(Q)^m}\fty \ d\mu(\vy)\right)^q \int_Q u^q
\ d\mu.$$
We may now proceed as in the case $q>1$, with $\mathcal{C}^k$ and
$Q_{k,j}$ defined exactly as above.  We assume again that $k_1=-\infty;$ as before, the extra term that appears when $k_1\neq-\infty$ can be handled in the same way as the terms in the sum in $k$ and $j$. Then
\begin{align*}
& \int_X (u\mathcal{T}\vf )^q \ d\mu \leq \sum_{Q\in \cd}
\left(\vp(B(Q))\int_{B(Q)^m}\fty \ d\mu(\vy)\right)^q \int_Q u^q
\ d\mu \\
&=   \sum_{Q\in \cd} \left(\vp(B(Q))\mu(B(Q))^m \prod_{i=1}^m
\dashint_{B(Q)}f_i \ d\mu \right)^q \int_Q u^q \ d\mu  \\
&\leq \sum_k a^{(k+1)q}\sum_j \sum_{\stackrel{Q\in
\mathcal{C}^k}{Q\subset Q_{k,j}}} \vp(B(Q))^q\mu(B(Q))^{mq} \int_Q u^q \
d\mu\\
&\leq  c \sum_{k,j} \vp(B(Q_{k,j}))^q\mu(B(Q_{k,j}))^{mq}
\left(\prod_{i=1}^m \dashint_{B(Q_{k,j})}f_i \ d\mu
\right)^q\int_{Q_{k,j}} u^q \ d\mu\\
&=  c \sum_{k,j} \left[ \vp(B(Q_{k,j}))\mu(B(Q_{k,j}))^{m} \prod_{i=1}^m
\dashint_{B(Q_{k,j})}f_i \ d\mu \left(\dashint_{Q_{k,j}} u^q \
d\mu\right)^{1/q}\right]^q \mu(Q_{k,j}).
\end{align*}
where the second to last inequality follows from a slight adaption of
Lemma \ref{packing}. We now use H\"older's inequality with $tp_i'$ and
$(tp_i')'$ and condition \eqref{general2wq<1hm}.  Let $[u,\vec{v}]$
represent the finite quantity from \eqref{general2wq<1hm}, we obtain
\begin{align*}
& \int_X (u\mathcal{T}\vf )^q \ d\mu\\
&\leq c \sum_{k,j} \left[
\vp(B(Q_{k,j}))\mu(B(Q_{k,j}))^{m} \prod_{i=1}^m \left(\dashint_{
B(Q_{k,j})}v_i^{-tp_i'} \ d\mu\right)^{1/tp_i'} \left(\dashint_{Q_{k,j}} u^q \
d\mu\right)^{1/q}\right]^q \\
& \qquad \times \prod_{i=1}^m\left(\dashint_{B(Q_{k,j})}(f_iv_i)^{(tp_i')'} \
d\mu\right)^{q/(tp_i')'} \mu(Q_{k,j})\\
&\leq  c[u,\vec{v}]^q \sum_{k,j}\prod_{i=1}^m\left(\dashint_{
B(Q_{k,j})}(f_iv_i)^{(tp_i')'} \ d\mu\right)^{q/(tp_i')'} \mu(Q_{k,j})^{q/p}\\
&\leq  c[u,\vec{v}]^q \left( \sum_{k,j}\prod_{i=1}^m\left(\dashint_{
B(Q_{k,j})}(f_iv_i)^{(tp_i')'} \ d\mu\right)^{p/(tp_i')'}
\mu(Q_{k,j})\right)^{q/p}\\
&\leq c[u,\vec{v}]^q \left( \sum_{k,j}\prod_{i=1}^m\left(\dashint_{
B(Q_{k,j})}(f_iv_i)^{(tp_i')'} \ d\mu\right)^{p/(tp_i')'}
\mu(E_{k,j})\right)^{q/p}\\
&\leq  c[u,\vec{v}]^q \prod_{i=1}^m \left(\int_X
M_{(tp_i')'}(f_iv_i)^{p_i} d\mu\right)^{q/p_i}\\
&\leq c[u,\vec{v}]^q \prod_{i=1}^m \|f_iv_i\|_{L^{p_i}(\mu)}^q.
\end{align*}
Thus concluding the proof of the case $q\leq 1$.
\qed

\section{Multilinear potential operators in Orlicz Spaces}\label{secc:multiOrl}

The aim of this section is twofold. We will show that it is possible to substantially generalize conditions \eqref{general2wq>1hm} and \eqref{general2wq<1hm} by resorting to the theory of Orlicz spaces and we will introduce the natural multilinear counterparts to some linear weighted estimates in the context of Orlicz spaces studied in \cite{Pe95, PW01}. These multilinear estimates will allow for a strictly wider range of indices than in the linear case, see Theorem \ref{orlicz2w} and Remark \ref{rmk:orlp<1} below.

We briefly recall some basic facts about Orlicz spaces, and refer the reader to \cite{BS88} and \cite{RR91} for a detailed account of the spaces. A function $\psi:[0,\infty)\ra [0,\infty)$ is called a Young function if it is continuous, convex, increasing, $\psi(0)=0$ and $\psi(t)\ra \infty$ as $t\ra \infty$.   Moreover, we shall assume $\psi$ is normalized so that $\psi(1)=1$ and $\psi$ satisfies the doubling condition, namely there exists constants $C$ and $N$ such that
$$\psi(2t)\leq C\psi(t), \quad \text{for all }t\geq N.$$
For each such function $\psi$ there exists a complementary Young function, denoted $\overline{\psi}$, such that
$$t\leq \psi^{-1}(t)\overline{\psi}^{-1}(t)\leq 2t, \quad t > 0.$$
The Orlicz space $L_\psi=L_\psi(X,\mu)$ is the class of all functions such that
$$\int_X \psi\left(\frac{|f(y)|}{\lambda}\right) d\mu(y) <\infty$$
for some $\lambda>0$.  The space $L_\psi$ is a Banach space equipped with the norm,
$$\|f\|_{\psi}=\inf\left\{\lambda >0: \int_X \psi\left(\frac{|f|}{\lambda}\right) \ d\mu \leq 1\right\}.$$
The space $L_{\overline{\psi}}$ is called \emph{the conjugate space of} $L_\psi$. Orlicz spaces satisfy the generalized H\"older inequality
$$\int_X |fg| \ d\mu \leq c\|f\|_\psi \|g\|_{\overline{\psi}}.$$
Notice that if $\psi(t)=t^r$ for $r\geq 1$ then $L_\psi=L^r(X, d\mu)$ and the complementary function $\overline{\psi}(t)=t^{r'}$ with conjugate space $L_{\overline{\psi}}=L^{r'}(X, d\mu)$.  Other interesting examples include $\psi(t)=t^r[\log(1+ t)]^{-1-\ep}$ for which the complementary is $\overline{\psi}(t)=t^{r'}[\log(1+t)]^{(r'-1)(1+\ep)}$.

Given a ball $B\subset X$ we define the $L_\psi$ average over $B$ by
$$\|f\|_{\psi,B}=\inf\left\{\lambda>0: \dashint_B \psi\left(\frac{|f|}{\lambda}\right) \ d\mu\leq 1\right\}.$$
Once we have defined an average over a single ball we may define a corresponding maximal function by
$$M_\psi f(x) = \sup_{B: x \in B} \|f\|_{\psi,B}$$
where the supremum is over all balls $B$ that contain $x$.
Notice that when $\psi(t)=t^r$ we have
$$\|f\|_{\psi,B}=\left(\dashint_B |f|^r \ d\mu\right)^{1/r},$$
and hence $M_\psi f(x) = M_\mu(|f|^r)^{1/r}$.  Furthermore, in this case,
$$M_\psi :L^p(X, d\mu)\ra L^p(X, d\mu)$$
if and only if $p>r$.  For a general $\psi$, P\'erez and Wheeden \cite{PW01} established the following characterization:
$$\int_X (M_\psi f)^p d\mu \leq C\int_X |f|^p \ d\mu$$
for all $f\in L^p(X, d\mu)$ if and only if there is a constant $c>0$ such that
\begin{equation} \label{Bpcond}\int_c^\infty \frac{\psi(t)}{t^p} \frac{dt}{t}\approx \int_c^\infty \left(\frac{t^{p'}}{\overline{\psi}(t)}\right)^{p-1} \frac{dt}{t}<\infty.\end{equation}
In the context of Orlicz spaces we have

\begin{theor}\label{orlicz2w} Suppose that $1<p_1, \cdots,
 p_m<\infty,$ $\frac{1}{p}=\frac{1}{p_1}+\cdots+\frac{1}{p_m}$ and $\frac{1}{m}<p\le q<\infty.$
Let $(X,\rho,\mu)$ be a space of homogeneous type, $K$ is a kernel such that \eqref{growth} holds with $\vp$ satisfying \eqref{mainassump} and $\Psi,\Phi_1,\ldots, \Phi_m$ be Young functions satrisfying
\begin{equation} \label{orliczq'}\int_c^\infty \left(\frac{t^q}{\Psi(t)}\right)^{q'-1}\frac{dt}{t}<\infty,\end{equation}
and
\begin{equation}\label{orliczp}\int_c^\infty \left(\frac{t^{p_i'}}{\Phi_i(t)}\right)^{p_i-1}\frac{dt}{t}<\infty \qquad 1\leq i\leq m\end{equation}
for some $c>0$.  Furthermore, let $u, v_k,$
$k=1,\cdots,m$ be weights defined on $X$ that satisfy condition
\eqref{orlicz2wq>1hm} if $q>1$ or condition \eqref{orlicz2wq<1hm} if $q\le 1,$
where
\begin{equation}\label{orlicz2wq>1hm}
\sup_{B\, \rho\text{-ball}} \vp(B) \mu(B)^{1/q + 1/{p_1}'+\cdots +1/{p_m'}}
\|u\|_{\Psi,B} \prod_{j=1}^m \|v_i^{-1}\|_{\Phi_i,B} < \infty;
\end{equation}
\begin{equation}\label{orlicz2wq<1hm}
\sup_{B\, \rho\text{-ball}} \vp(B) \mu(B)^{1/q + 1/{p_1}'+\cdots+1/{p_m}'}
\left( \dashint_B u^{q} d\mu \right)^{1/q} \prod_{j=1}^m
\|v_i^{-1}\|_{\Phi_i,B}<
\infty.
\end{equation}
Then there exists a constant $C$  such that
\[
\left(\int_{X}\left(\abs{\mathcal{T}\vec{f}}u\right)^q\,d\mu\right)^{1/q}\le
C \prod_{k=1}^m\left(\int_X(\abs{f_k}v_k)^{p_k}\,d\mu\right)^{1/p_k}
\]
for all $\vec{f}\in L^{p_1}(X,v_1^{p_1})\times\cdots\times
L^{p_m}(X,v_m^{p_m}).$
\end{theor}
\begin{remark} Theorem \ref{general2w} is contained in Theorem \ref{orlicz2w} since it corresponds to
$$\Psi(t)= t^{rq}, \Phi_1(t)=t^{rp_1'},\ldots, \Phi_m(t)=t^{rp_m'}$$
whose complementary functions satisfy conditions \eqref{orliczq'} and \eqref{orliczp}.
 \end{remark}
 We provide a brief sketch of the proof of Theorem \ref{orlicz2w} when $q>1$ and $\diam_\rho(X) = \infty$.  The proof when $q<1$ will be similar to that of Theorem \ref{general2w}.
 \begin{proof}[Proof of Theorem \ref{orlicz2w}]
 The same decomposition techniques as in the proof of Theorem \ref{general2w} yield
 \bey
 \int_X (\mathcal{T}\vf) g u\ d\mu &\leq&  C\sum_{k,j}  \vp(B(Q_{k,j}))\mu(B(Q_{k,j}))^m \prod_{i=1}^m \dashint_{ B(Q_{k,j})}f_i(y_i) \ d\mu(y_i)  \\
&&\qquad  \quad \times \dashint_{Q_{k,j}} gu \ d\mu \ \ \mu(Q_{k,j}).
\eey
Using the generalized H\"older inequality for Orlicz spaces we have
 \begin{align*}
 & \int_X (\mathcal{T}\vf) g u\ d\mu\\
  &\leq  C\sum_{k,j}  \vp(B(Q_{k,j}))\mu(B(Q_{k,j}))^m \prod_{i=1}^m \|f_iv_i\|_{\overline{\Phi}_i,B(Q_{k,j})} \|v_i^{-1}\|_{\Phi_i, B(Q_{k,j})} \\
&\qquad  \quad \times \|g\|_{\overline{\Psi}, B(Q_{k,j})} \|u \|_{\Psi, B(Q_{k,j})}\ \ \mu(Q_{k,j})\\
&\leq  c\sum_{k,j} \prod_{i=1}^m \|f_iv_i\|_{\overline{\Phi}_i,  B(Q_{k,j})} \|g\|_{\overline{\Psi},B(Q_{k,j})} \mu(Q_{k,j})^{1/q'+1/p}\\
&\leq  c\left(\sum_{k,j} \left(\prod_{i=1}^m \|f_iv_i\|_{\overline{\Phi}_i,  B(Q_{k,j})}\right)^q \mu(Q_{k,j})^{q/p}\right)^{1/q}\left(\sum_{k,j} \|g\|_{\overline{\Psi}, B(Q_{k,j})}^{q'} \mu(Q_{k,j})\right)^{1/q'}\\
&\leq  c\left(\sum_{k,j} \left(\prod_{i=1}^m \|f_iv_i\|_{\overline{\Phi}_i,  B(Q_{k,j})}\right)^p \mu(E_{k,j})\right)^{1/p}\left(\sum_{k,j} \|g\|_{\overline{\Psi}, B(Q_{k,j})}^{q'} \mu(E_{k,j})\right)^{1/q'}\\
&\leq  c \left( \prod_{i=1}^m \left(\int_X  M_{\overline{\Phi}_i}( f_iv_i)^{p_i} \ d\mu\right)^{1/p_i} \right) \left( \int_X (M_{\overline{\Psi}} g)^{q'} \ d\mu\right)^{1/q'}\\
&\leq  C \left(\prod_{i=1}^m \|f_iv_i\|_{L^{p_i}(\mu)}\right) \|g\|_{L^{q'}(\mu)} \
\end{align*}
where the last line follows since $\overline{\Psi},\overline{\Phi}_1,\ldots,\overline{\Phi}_m$ satisfy \eqref{orliczq'} and \eqref{orliczp} so
$$\left\{ \begin{array}{ccc} M_{\overline{\Psi}} : L^{q'}(X, d\mu) &\ra& L^{q'}(X, d\mu) \\ M_{\overline{\Phi}_1}:L^{p_1}(X, d\mu)& \ra& L^{p_1}(X, d\mu)\\
 & \vdots & \\
 M_{\overline{\Phi}_m}:L^{p_m}(X, d\mu)& \ra& L^{p_m}(X, d\mu) .\end{array} \right.
 $$

\end{proof}
We now give some applications of Theorem \ref{orlicz2w}.  For simplicity let $q=p$, and let $\Psi,\Phi_1,\ldots,\Phi_m$ be the Young functions defined by
$$\Psi(t)=t^{p}(\log(1+t))^{p-1+\ep}, \Phi_1(t) =t^{p_1'}(\log(1+t))^{p_1'-1+\ep}, \ldots, \Phi_m(t)= t^{p_m'}(\log(1+t))^{p_m'-1+\ep}.$$
Notice that these functions satisfy conditions  \eqref{orliczq'} and \eqref{orliczp}.  We denote the Orlicz spaces as
$$L_\Psi=L^p(\log L)^{p-1+\ep}, L_{\Phi_1}=L^{p_1'}(\log L)^{p_1'-1+\ep}, \ldots, L_{\Phi_m}=L^{p_m'}(\log L)^{p_m'-1+\ep}.$$
Thus as a corollary we have the following result.
\begin{corollary}  Suppose that $1<p_1, \ldots, p_m<\infty,$ $\frac{1}{p}=\frac{1}{p_1}+\cdots+\frac{1}{p_m}$, $(X,\rho,\mu)$ is a space of homogeneous type, and $K$ is a kernel such that \eqref{growth} holds with $\vp$ satisfying \eqref{mainassump}.  Let $\ep>0$ and $u,v_1,\ldots, v_m$ be weights that satisfy
\begin{equation}\label{LlogLp>1} \sup_B \vp(B) \mu(B)^{1/q+1/p_1'+\cdots + 1/p_m'} \|u\|_{L^p(\log L)^{p-1+\ep},B} \prod_{i=1}^m \|v_i^{-1}\|_{L^{p_i'}(\log L)^{p_i'-1+\ep},B} <\infty\end{equation}
if $p>1$, or
\begin{equation}\label{LlogLp<1}\sup_B \vp(B) \mu(B)^{1/q+1/p_1'+\cdots + 1/p_m'} \left(\dashint_B u^p \ d\mu\right)^{1/p} \prod_{i=1}^m \|v_i^{-1}\|_{L^{p_i'}(\log L)^{p_i'-1+\ep},B} <\infty\end{equation}
if $p\leq 1$.  Then
\[
\left(\int_{X}\left(\abs{u\mathcal{T}\vec{f}}\right)^p\,d\mu\right)^{1/p}\le
C \prod_{i=1}^m\left(\int_X(\abs{f_i}v_i)^{p_i}\,d\mu\right)^{1/p_i}
\]
for all $\vec{f}\in L^{p_1}(X,v_1^{p_1}d\mu)\times\cdots\times
L^{p_m}(X,v_m^{p_m}d\mu).$
\end{corollary}
We now use these $L(\log L)$ results to obtain a different estimate in terms of the fractional maximal function of a weight.  Let $\gamma$ be a functional on the balls of $X$ and define $M_\gamma$ as in \cite{PW01} by
$$M_\gamma f(x) =\sup_{B : x \in B} \gamma(B)\int_B |f| \ d\mu.$$
When $X=\rr^n$, $\gamma(B)=|B|^{\al/n-1}$ corresponds to the fractional maximal operator $M_\al$.  Let $M^k$ denote the $k$-th iterate of the Hardy-Littlewood maximal operator, i.e., $M^k=M_\mu \stackrel{(k \ {\rm times})}{\circ \cdots \circ} M_\mu$.  Also, for $1<p<\infty$, $[p]$ will denote the  greatest integer less than or equal to $p$.  We have the following result.
\begin{corollary}\label{fracmax} Let $1<p_1,\ldots, p_m<\infty$, $1/p=1/p_1+\cdots+1/p_m$, and $\mathcal{T}$, $K$, and $\vp$ be as in Theorem \ref{general2w}.  Furthermore, suppose $0<\al_1,\cdots,\al_m<1$ and $\al_1+\cdots +\al_m=1$, and $\tilde{\vp}_i$ is the functional: $B\mapsto (\vp(B)^{\al_i}\mu(B))^{p_i}/\mu(B)$ and $w$ be any weight.  Then if $p> 1$ we have
\begin{equation} \label{fracmaxp>1} \left(\int_X |\mathcal{T}\vf|^p w \ d\mu\right)^{1/p} \leq C\prod_{i=1}^m \left(\int_X |f_i|^{p_i} M_{\tilde \vp_i}(M^{[p]}w) \ d\mu\right)^{1/p_i}\end{equation}
and if $p\leq 1$
\begin{equation} \label{fracmaxp<1} \left(\int_X |\mathcal{T}\vf|^p w \ d\mu\right)^{1/p} \leq C\prod_{i=1}^m \left(\int_X |f_i|^{p_i} M_{\tilde \vp_i}(w) \ d\mu\right)^{1/p_i}.
\end{equation}
\end{corollary}
Before we present the proof of Corollary \ref{fracmax} a few remarks are in order.

\begin{remark}\label{rmk:orlp<1} Inequalities \eqref{fracmaxp<1} and \eqref{fracmaxp>1} are new even in the Euclidean setting and they constitute the multilinear counterparts to the linear ones in \cite[Theorem 2.5]{PW01}.

Also, in the Euclidean setting and when $\tilde{\vp_i}(B)\approx r(B)^{p_i\al_i}/|B|$, $i=1, \ldots, m$ for $\al_1+\cdots + \al_m =\al$ (i.e., $\mathcal{T}$ is the multilinear fractional integral operator $\mathcal{I}_\al$), inequalities \eqref{fracmaxp<1} and \eqref{fracmaxp>1} read as follows: If $p>1$
\begin{equation}\label{carlosp>1}\left(\int_{\rr^n} |\mathcal{I}_\al\vf|^p w \ dx\right)^{1/p} \leq C\prod_{i=1}^m \left(\int_{\rr^n} |f_i|^{p_i} M_{p_i\al_i}(M^{[p]}w) \ dx \right)^{1/p_i}\end{equation}
and if $p\leq 1$
\begin{equation}\label{carlosp<1}\left(\int_{\rr^n} |\mathcal{I}_\al\vf|^p w \ dx\right)^{1/p} \leq C
\prod_{i=1}^m \left(\int_{\rr^n} |f_i|^{p_i} M_{p_i\al_i}(w) \ dx\right)^{1/p_i}. \end{equation}

In turn, \eqref{carlosp>1} and \eqref{carlosp<1} arise as the multilinear versions of the linear inequalities of the form
\begin{equation} \label{CP} \int_{\rr^n} |I_\al f|^p \ w\ dx \leq C \int_{\rr^n} |f|^p M_{\al p}(M^{[p]} w) \ dx,
\end{equation}
which were addressed in \cite{Pe95} for $p > 1$.  It must be observed that in the linear case ($m=1$), inequality \eqref{carlosp<1} with $p=1$ is false, see \cite[Theorem 2.1]{CPSS05}. Therefore, if $m > 1$, inequality \eqref{carlosp<1} (and, more generally, inequality \eqref{fracmaxp<1}) allows for a range of indices forbidden in the linear case.

Finally, notice that inequality \eqref{carlosp>1} does not follow (at least directly) from the fact that,
$$\mathcal{I}_\al \vf \leq I_{\al_1} f_1 \cdots I_{\al_m} f_m.$$
Indeed, if one uses this product bound, followed by H\"older's inequality and \eqref{CP}, one obtains
\begin{eqnarray}\left(\int_{\rr^n} |\mathcal{I}_\al\vf|^p w \ dx\right)^{1/p} &\leq& \left(\int_{\rr^n} \Pi_{i=1}^m |I_{\al_i} f_i| w \ dx \right)^{1/p}\nonumber \\
&\leq & \prod_{i=1}^m \left(\int_{\rr^n} |I_{\al_i} f_i|^{p_i} w \ dx\right)^{1/p_i}\nonumber\\
&\leq &  \prod_{i=1}^m \left(\int_{\rr^n} |f_i|^{p_i} M_{{p_i\al_i}}(M^{[p_i]}w) \ dx\right)^{1/p_i}.\label{lesssharp}
\end{eqnarray}
However, since $p<p_i$ inequality \eqref{carlosp>1} is sharper than \eqref{lesssharp}.
\end{remark}

\begin{proof}[Proof of Corollary \ref{fracmax}]
In order to prove Corollary \ref{fracmax} we will show that there exists $\ep>0$ such that the weights
$$u=w^{1/p}, v_1= M_{\tilde \vp_1}(w)^{1/p_1},\ldots, v_m=M_{\tilde \vp_m}(w)^{1/p_m}$$
satisfy \eqref{LlogLp<1} if $p\leq 1$ or the weights
$$u=w^{1/p}, v_1= M_{\tilde \vp_1}(M^{[p]}w)^{1/p_1},\ldots, v_m=M_{\tilde \vp_m}(M^{[p]}w)^{1/p_m}$$
satisfy \eqref{LlogLp>1} if $p>1$.  We start with the case $p\leq1$.  Notice that for any ball $B$ and $x\in B$ we have
$$M_{\tilde \vp_i}(w)(x) \geq \frac{(\vp(B)^{\al_i} \mu(B))^{p_i}}{\mu(B)} \int_B w \ d\mu.$$
Hence,
$$\prod_{i=1}^m \|M_{\tilde \vp_i}(w)^{-1/p_i}\|_{L^{p_i'}(\log L)^{p_i'-1+\ep},B} \leq \frac{1}{\vp(B) \mu(B)^m}\prod_{i=1}^m \left(\dashint_B w \ d\mu\right)^{-1/p_i}$$
from which \eqref{LlogLp<1} follows.  Now for the case $p> 1$.  Let $\delta>0$ and $\hat{B}=(1+\delta)\kappa B$.  For any $x\in B$ we have
$$\left(M_{\tilde \vp_i}(M^{[p]}w)(x)\right)^{-1/p_i} \leq \frac{1}{\vp(\hat{B})^{\al_i} \mu(\hat{B})} \left( \dashint_{\hat{B}} M^{[p]} w \ d\mu\right)^{-1/p_i}.$$
Hence,
\bey \lefteqn{\vp(B) \mu(B)^{m} \|w^{1/p}\|_{L^p(\log L)^{p-1+\ep},B} \prod_{i=1}^m \|M_{\tilde{\vp}_i}(M^{[p]}w)^{-1/p_i}\|_{L^{p_i'}(\log L)^{p_i'-1+\ep},B}} \\
&\leq & C\|w^{1/p}\|_{L^p(\log L)^{p-1+\ep},B} \prod_{i=1}^m\left( \dashint_{\hat{B}} M^{[p]} w \ d\mu\right)^{-1/p_i}\\
&=&C\|w^{1/p}\|_{L^p(\log L)^{p-1+\ep},B}\left(\dashint_{\hat{B}} M^{[p]}w \ d\mu\right)^{-1/p},
\eey
where we have used the reverse doubling properties of $\vp$ and $\mu$ (see \eqref{reversedoubling} and \eqref{doubling}).  Choosing $\ep=[p]-p+1>0$ we have
$$\|w^{1/p}\|^p_{L^p(\log L)^{p-1+\ep},B}=\|w\|_{L(\log L)^{p-1+\ep},B}=\|w\|_{L(\log L)^{[p]},B}.$$
However, Lemma 8.5 in \cite{PW01} shows that for any $\delta>0$,
$$\|w\|_{L(\log L)^{[p]},B} \leq C\dashint_{\hat{B}} M^{[p]} w \ d\mu.$$
\end{proof}

\section{Proof of Theorem~\ref{poincareineqTheorem2w}}\label{secc:proofmain}

Let $\Omega$ be an open connected subset of $\rn$ and
$Y=\{Y_k\}_{k=1}^M$ a family of real-valued, infinitely
differentiable vector fields. We identify the ${Y_j}$'s with the first order
differential operators acting on Lipschitz functions  defined on
$\Omega$ by the formula
\[
Y_kf(x)=Y_k(x)\cdot \nabla f(x), \quad k=1,\cdots,M,
\]
and we set $Yf=(Y_1 f, Y_2f,\cdots, Y_Mf)$ and
\[
\abs{Yf(x)}=\left(\sum_{k=1}^M\abs{Y_k f(x)}^2\right)^{1/2}, \quad
x\in \Omega.
\]
\begin{definition} Let $\Omega$ and $Y$ be as above. $Y$ is said to satisfy H\"ormander's condition in
$\Omega$ if there exists an integer $M_0$ such that the family of
commutators of vector fields in $Y$ up to length $M_0,$ i.e., the family if vector fields $Y_1$, $Y_2,\cdots$, $Y_{M_0},$ $[Y_{k_1} Y_{k_2}],\cdots,[Y_{k_1},[Y_{k_2},[\cdots,Y_{k_{M_0}}]]\cdots],$ span  $\rn$ at every point of
$\Omega.$
\end{definition}

Suppose that $Y=\{Y_k\}_{k=1}^M$ satisfies H\"ormander's
condition in $\Omega.$  Let $C_Y$ be the family of absolutely
continuous curves $\zeta:[a,b]\to\Omega,$ $a\le b,$ such that
there exist measurable functions $c_j(t),$ $a\le t\le b,$
$j=1,\cdots, M,$ satisfying $\sum_{j=1}^{M} c_j(t)^2\le 1$ and
$\zeta'(t)=\sum_{j=1}^{M}c_j(t) Y_j(\zeta(t))$ for almost every
$t\in [a,b].$ If $x,\,y\in\Omega$ define
\[\rho(x,y)=\inf\{T>0:\text{ there exists }\zeta \in C_Y \text{ with } \zeta(0)=x \text{ and }\zeta(T)=y\}.\]
The function $\rho$ is in fact a metric in $\Omega$ called the
Carnot-Carath\'eodory metric associated to $Y.$  A detailed study of
the geometry of Carnot-Carath\'eodory spaces can be found in
Nagel-Stein-Wainger~\cite{NSW}.

\begin{remark}\label{rhodoubling} Let $Y$ satisfy
H\"ormander's condition in $\Omega$ with integer $M_0$ and let
$\rho$ be the associated Carnot-Carath\'eodory metric. Nagel-Stein-Wainger~\cite{NSW} proved that for
every compact set $K\subset\Omega$ there exist positive constants  $R_0,$  $C,$ $C_1$ and $C_2$ depending on $K$ such that
\[
\abs{B_\rho(x,2r)}\le C\,\abs{B_\rho(x,r)}, \qquad x\in K, \,r< R_0
\]
and
\[
C_1\abs{x-y}\le \rho(x,y)\le C_2\abs{x-y}^{1/{M_0}}, \qquad x,\,y\in K.
\]
When $\Omega$ is bounded, as noted by Bramanti-Brandolini~\cite[p.534]{BB05}, the last inequality implies that one can actually take $R_0=\infty$.
Moreover, Bramanti-Brandolini~\cite[p.533]{BB05} proved that if $B$ is a $\rho$-ball contained in $K$ and $\Omega$ is bounded then there exists a positive constant $C=C_{\Omega,K, Y}$ such that
\begin{equation}\label{Bdoubling}
\abs{B_\rho(x, r)\cap B}\ge C\abs{B_\rho(x,r)}, \qquad x\in B, \, 0<r<\diam_\rho(B).
\end{equation}
As a consequence, the triple $(B,\rho, \text{Lebesgue measure})$ becomes a space of homogeneous type for all $\rho$-balls $B$ contained in $K$ and with uniform doubling constants that depend on $\Omega,$ $K$ and $Y$. It is often found in the literature the claim that for any compact set $K \subset \Omega$ the triple $(K, \rho,  \text{Lebesgue measure})$ constitutes a space of homogeneous type. However, that is not true in general. Indeed, just in the Euclidean setting, simple examples where $K \subset B(0, 1) \subset \rn$ has an exponentially pronounced cusp will show so. On the other hand, some regularity properties for $\partial K$ will ensure, in general, that $(K, \rho,  \text{Lebesgue measure})$ is a space of homogeneous type. One such property, and one that every $\rho$-ball $B \subset \Omega$ does posses, is expressed by inequality \eqref{Bdoubling}.
\end{remark}

\begin{remark}\label{reversedoubling2}  A reverse doubling property in the context of Carnot-Carath\'eodory spaces was proved by
Franchi-Wheeden~\cite[pp.82-89]{FW99}. More precisely, let $Y$ be a collection of vector fields satisfying H\"ormander's condition in $\Omega$  and let
$\rho$ be the associated Carnot-Carath\'eodory metric. If $\Omega_0\subset\subset\Omega$ is an open bounded set   and $\tau>1,$ then there are positive constants $R_0=R_0(\Omega, \Omega_0,Y)$ and $C=C(\Omega,\Omega_0,Y)$ such that if $B$ is a $\rho$-ball  contained in ${\Omega_0}$ of radius smaller that $R_0$ and $B_1$ and $B_2$ are $\rho$-balls such $B_1 \subset B_2\subset \tau B$ then
\begin{equation}\label{RDorder1}
\frac{|B_2|}{|B_1|} \geq C \,\frac{r(B_2)}{r(B_1)}.
\end{equation}
\end{remark}

\bigskip

Let $\Omega_1\subset \rr^{n_1}$ and $\Omega_2\subset \rr^{n_2}$ be
open connected sets. Given  two families of vector fields $Y^{(1)}$
on $\Omega_1$ and $Y^{(2)}$ on $\Omega_2$ the union of the two sets
is defined as the collection $Y$ of vector fields defined on
$\Omega_1\times\Omega_2\subset \rr^{n_1+n_2}$ obtained by adjoining
zero coordinates appropriately to the vectors in $Y^{(1)}$ and
$Y^{(2)}$ to obtain vectors in $\rr^{n_1+n_2}.$ We note that if
$Y^{(1)}$ and $Y^{(2)}$ satisfy H\"ormander's condition in
$\Omega_1$ and $\Omega_2,$ respectively,  then so does  $Y$ in
$\Omega_1\times\Omega_2$. The following lemma
(Lu-Wheeden~\cite[Lemma 1]{LW98a}) describes the relation between
the Carnot-Carath\'eodory metrics associated to $Y^{(1)},$ $Y^{(2)}$
and $Y.$

\begin{lemma}\label{productdistance}
Let $d_1$ and $d_2$ be Carnot-Carath\'eodory metrics associated with
H\"ormander vector fields $Y^{(1)}$ and $Y^{(2)}$ in $\Omega_1$ and
$\Omega_2,$ respectively. Let $d$ be the metric in $\Omega_1\times\Omega_2$
associated with the union $Y$ of the  two collections.
Then if $x=(x_1,x_2)$ and $y=(y_1,y_2)$ are any two points in
$\Omega_1\times\Omega_2,$
\[
d(x,y)=\max\{d_1(x_1,y_1), d_2(x_2,y_2)\}.
\]
\end{lemma}

\begin{remark}\label{productball} Let $\vec{x}=(x_1,x_2)$ with $x_1\in\Omega_1$ and $x_2\in\Omega_2$ and $r>0.$
Lemma~\ref{productdistance} implies that $B_d(\vec{x},
r)=B_{d_1}(x_1,r)\times B_{d_2}(x_2,r).$
\end{remark}

\bigskip

We now deduce a multilinear  representation formula in
the setting of Carnot-Carath\'eodory spaces.
Theorem~\ref{poincareineqTheorem2w}
will follow from
Corollary~\ref{multrepformula} below and the weighted boundedness
properties of the multilinear fractional operators in
Corollary \ref{potopbound2w}.

\begin{theor}[Representation formula]\label{repformulatheo}
 Suppose Y is a collection of  vector fields on a connected bounded open set $\Omega\subset \rn$ satisfying
  H\"ormander's condition, $\rho$ is the associated Carnot-Carath\'eodory metric and $\Omega_0\subset\subset\Omega$ is an open  set. There exist positive constants
  $r_0=r_0(\Omega,\Omega_0,Y)$ and  $C_{\Omega,\Omega_0, Y}$ such that  for all $\rho$-ball $B\subset\overline{\Omega_0}$ with radius less than $r_0$ and for all $f\in C^1(\overline{B}),$
\begin{equation}\label{repformula}
\abs{f(x)-f_{B}}\le C_{\Omega,\Omega_0,Y}\,
\int_{B}\abs{Yf(y)}\frac{\rho(x,y)}{\abs{B_{\rho}(x,\rho(x,y))}}\,dy,
\qquad  x\in B.
\end{equation}
\end{theor}

\begin{proof}
This is essentially a consequence of Theorem 1 in Lu-Wheeden~\cite{LW98b}, we only need to check that the hypotheses (H1)-(H3) in that theorem hold true for any $\rho$-ball $B\subset\overline{\Omega_0}$ with radius sufficiently small and with constants depending only on $\Omega$ and  $Y$.

Hypothesis (H1): In our context (H1) can be stated as the existence of positive constants $a_1 \geq 1$  and $C_1 > 0$ such that for all $\rho$-balls $\tilde{B}$ with $a_1 \tilde{B} \subset B$
\begin{equation}\label{PoinJerison}
\int_{\tilde{B}} |f - f_{\tilde{B}}| \, dx \leq C_1 \,r(\tilde{B}) \int_{a_1 \tilde{B}} |Y f| \, dx,
\end{equation}
where, again, $r(\tilde{B})$ denotes the radius of $\tilde{B}.$ Inequality \eqref{PoinJerison} holds true as a consequence of Jerison's Poincar\'e estimate in \cite{J86} if the radius of $B$ is sufficiently small. More precisely, for every compact set
$K\subset \Omega$ there are constants $C_{K,Y}$ and $r_{K,Y}$ such that for $u\in \text{Lip}(\tilde{B})$
\[
\int_{\tilde{B}} |f - f_{\tilde{B}}| \, dx \leq C_{K,Y}\, r(\tilde{B}) \int_{2 \tilde{B}} |Y f| \, dx,
\]
whenever $\tilde{B}$ is a $\rho$-ball centered at $K$ and radius $r(\tilde{B})<r_{K,Y}$ (see Haj{\l}asz-Koskela~\cite[Theorem 11.20]{HK00}).
  Then (H1) follows with  $a_1=2$ and $C_1=C_{\overline{\Omega_0},Y}$ if we choose $K=\overline{\Omega_0}$ and $B\subset \overline{\Omega_0}$ with radius smaller than $r_{\overline{\Omega_0},Y}.$  In fact, Jerison proved the inequality with the $L^2$ norms on both sides, but the same arguments work with the $L^1$ norm. He also proved that one can take $a_1=1.$

 Hypothesis (H2): (H2) is the reverse doubling condition \eqref{RDorder1}, which, as mentioned before, was proved in
Franchi-Wheeden~\cite[pp.82-89]{FW99}.

Hypothesis (H3): (H3) is the `segment property' for every ball $\tilde{B}\subset B,$ which holds for any metric induced by a collection of Carnot-Carath\'eodory vector fields  (see Franchi-Wheeden~\cite[p.66 and Example 3 (p.82)]{FW99} and note that the notion of `segment property' in Franchi-Wheeden\cite{FW99} easily implies the notion of `segment' property in Lu-Wheeden ~\cite[p.580]{LW98b}).

We have then checked that hypotheses (H1)-(H3) hold  with constants
depending only on $\Omega_0$ and $Y$  and for any $\rho$-ball
$B\subset\overline{\Omega_0}$ with radius smaller than $r_0$
defined as the minimum of the upper bounds obtained for the radii in
(H1) and (H2). Therefore, the representation formula
\eqref{repformula} holds true by Theorem 1 in
Lu-Wheeden~\cite{LW98b}.

It must be noticed that \eqref{repformula} holds true for every $x
\in B$ and not only for a.e. $x \in B.$ The proof of Theorem 1
in Lu-Wheeden~\cite{LW98b} depends on the representation formula of
Lemma 3 in   Lu-Wheeden~\cite{LW98b}. Lemma 3, in turn, is based on
Theorem 1 in Franchi-Wheeden~\cite{FW99} and a close examination of
its proof shows that, in our case, it actually holds for every $x
\in B$ since we are assuming that $f$ is continuous and therefore
every point in $B$ is a Lebesgue point of $f$. Another
explanation for the fact that \eqref{repformula} holds for every
$x\in B$ would be that from Lu-Wheeden~\cite{LW98b}, the reasoning
above gives that \eqref{repformula} holds a.e. in $B.$ This and
the fact that both sides of the inequality are continuous in $x$
give that \eqref{repformula} holds for every $x\in B.$

\end{proof}

\begin{corollary}[Multilinear representation formula]\label{multrepformula}
 Suppose $Y$ is a collection of vectors fields on an open bounded connected set $\Omega\subset \rn$ satisfying
H\"omander's condition and $\tilde{Y}$ is the vector field defined on $\Omega^m$ that is the union of $m$ copies of $Y.$ Denote by $\rho$
the  Carnot-Carath\'eodory metric in $\Omega$ associated to $Y$ and by $\tilde{\rho}$
 the Carnot-Carath\'eodory metric in $\Omega^{m}$ associated to $\tilde{Y}.$  Let $\Omega_0\subset\subset\Omega$ be an open  set. There exist positive constants
  $r_0=r_0(\Omega,\Omega_0,Y)$ and  $C_{\Omega,\Omega_0, Y}$ such that  for all $\rho$-ball $B\subset{\Omega_0}$ with radius less than $r_0$ and for all $f_k\in C^1(\overline{B}),$ $k=1,\cdots, m,$
 \begin{equation}\label{repformula2}
\abs{\prod_{k=1}^m f_k(x)- \prod_{k=1}^m {f_k}_B}\le C_{\Omega,\Omega_0,Y}\,
\sum_{k=1}^m \mathcal{I}_{B,1}(f_1\chi_B, \cdots, (Yf_k)\chi_B, \cdots,
f_m\chi_B)(x),
\end{equation}
for all $x\in B.$
 \end{corollary}

 \begin{proof}
Let $r_0$ be given  by Theorem~\ref{repformulatheo} when applied to $\Omega^m,$ $\Omega_0^m$ and $\tilde{Y}.$ If $B$ is a $\rho$-ball of radius less than $r_0$ contained in $\Omega_0,$  by Lemma~\ref{productdistance}, $B^{m}$ is a
$\tilde{\rho}$-ball contained in $\Omega_0^{m}$ of radius less than $r_0.$
Theorem~\ref{repformulatheo} with $f(\vec{y})=\prod_{k=1}^m f_k(y_k)$ gives
\begin{align*}
\abs{f(\vec{x})-\frac{1}{\abs{B^m}}\int_{B^m}f(\vec{y})\,d\vec{y}}&\le
C_{\Omega,\Omega_0,Y}\, \int_{B^m}\abs{\tilde{Y}f(\vec{y})}
\frac{\tilde{\rho}(\vec{x},\vec{y})}{\abs{B_{\tilde{\rho}}(\vec{x},
\tilde{\rho}(\vec{x}, \vec{y}))}}\,d\vec{y}\\
&=C_{\Omega,\Omega_0,Y}\, \int_{B^m}\abs{\tilde{Y}f(\vec{y})}
\frac{\tilde{\rho}(\vec{x},\vec{y})}{\prod_{k=1}^m\abs{B_\rho(x_k,
\tilde{\rho}(\vec{x}, \vec{y}))}}\,d\vec{y},
\end{align*}
for all $\vec{x}=(x_1,\cdots, x_m)\in B^m.$
Taking   $x_1=\cdots=x_m=x\in B$  we obtain
\begin{align*}
&\abs{\prod_{k=1}^m f_k(x)-\prod_{k=1}^m {f_k}_B}\le C_{\Omega,\Omega_0,Y}\,
\int_{B^m}\abs{\tilde{Y}f(\vec{y})}
\frac{\tilde{\rho}(\vec{x},\vec{y})}{\abs{B_\rho(x,
\tilde{\rho}(\vec{x},
\vec{y}))}^m}\,d\vec{y}\\
&\lesssim \int_{B^m}\sum_{k=1}^m \abs{ f_1(y_1)\chi_B(y_1)
\cdots (Yf_k(y_k))\chi_B(y_k)\cdots f_m(y_m)\chi_B(y_m)}
\frac{\tilde{\rho}(\vec{x},\vec{y}) }{\abs{B_\rho(x,
\tilde{\rho}(\vec{x},
\vec{y}))}^m}\,d\vec{y},
\end{align*}
where the constants depend only on $\Omega,$ $\Omega_0$ and $Y$. Since $\tilde{\rho}(\vec{x},\vec{y})\sim \rho(x,\vec{y}),$  it follows from Remark~\ref{rhodoubling}  that $\abs{B_\rho(x,\tilde{\rho}(\vec{x}, \vec{y}))}\sim\abs{B_\rho(x, \rho(x,\vec{y}))}$ uniformly for $x\in\Omega_0,$ and therefore we obtain  \eqref{repformula2}.

\end{proof}

\begin{proof}[Proof of Theorem~\ref{poincareineqTheorem2w}] Let $r_0$ be the minimum of the radii given by Remark~\ref{reversedoubling2}
 and by Corollary~\ref{multrepformula} when applied to $\Omega,$ $\Omega_0$ and $Y$ as in the hypotheses of Theorem~\ref{poincareineqTheorem2w}.
  Let $B$ be a $\rho$-ball contained in $\Omega_0$ of radius smaller than  $r_0.$
The Poincar\'e inequality $\eqref{poincareineq}$ will follow from the multilinear representation formula \eqref{repformula2} once we have checked that $(B,\rho,\text{Lebesgue measure})$ satisfies the hypotheses of Corollary~\ref{potopbound2w} with uniform constants depending only on $\Omega,$ $\Omega_0$ and $Y$.

As noted in Remark~\ref{rhodoubling}, $(B,\rho, \text{Lebesgue measure})$ is a space of homogeneous type with doubling constant uniform in $B$ depending on $\Omega,$ $\Omega_0$ and $Y.$

The reverse doubling condition \eqref{reversedoubling} in this context  means that there are positive  constants $c$ and $\delta,$ depending only on $\Omega,$ $\Omega_0$ and $Y,$   such that
\[
\frac{\abs{B_\rho(x_1,r_1)\cap B}}{\abs{B_\rho(x_2,r_2)\cap B}}\ge c \left(\frac{r_1}{r_2}\right)^\delta,
\]
whenever $B_\rho(x_2,r_2)\subset B_\rho(x_1,r_1),$ $x_1,x_2\in B,$ and $0<r_1,\,r_2\le 2\,\diam_\rho(B)$.
By \eqref{Bdoubling}, this reduces to prove that there are positive constants $c$ and $\delta,$ depending only on $\Omega,$ $\Omega_0$ and $Y,$ such that
\begin{equation}
\frac{\abs{B_\rho(x_1,r_1)}}{\abs{B_\rho(x_2,r_2)}}\ge c \left(\frac{r_1}{r_2}\right)^\delta,
\end{equation}
whenever $B_\rho(x_2,r_2)\subset B_\rho(x_1,r_1),$ $x_1,x_2\in B,$ and $0<r_1,\,r_2\le 2\, \diam_\rho(B)$.
This is the result proved by Franchi-Wheeden~\cite{FW99}  with $\delta=1$ and $\tau=5$ as indicated in Remark~\ref{reversedoubling2}.

As explained in Remark~\ref{mainassumpremark},  \eqref{mainassump} is satisfied with $\ep=\alpha$ and  $C_2$ depends only on structural constants independent of $B$ if $C_1=2\kappa a_1 A m$ as given in the proof of Theorem~\ref{general2w} at the moment of applying the result of Lemma~\ref{packing}.

The growth condition \eqref{growth} for the kernel \eqref{kerneltildam} in this context means that for
every positive constant $C_1$
there exists a positive constant $C_2=C_2(\Omega,\Omega_0,Y)$ such that for all $\vec{x},\,\vec{y},\,\vec{z}\in B^m,$
\begin{eqnarray*}
\frac{{\rho}(\vec{x}, \vec{y})}{\prod_{k=1}^m \abs{B_{{\rho}}(x_k, {\rho}(\vec{x},\vec{y}))\cap B}}&\le& C_2\, \frac{{\rho}(\vec{z}, \vec{y})}{\prod_{k=1}^m\abs{B_{\rho}(z_k, {\rho}(\vec{z},\vec{y}))\cap B}}, \quad {\rho}(\vec{z}, \vec{y})\le C_1 \,{\rho}(\vec{x},\vec{y})\\
\frac{{\rho}(\vec{x}, \vec{y})}{\prod_{k=1}^m\abs{B_{{\rho}}(x_k, {\rho}(\vec{x},\vec{y}))\cap B}}&\le& C_2\, \frac{{\rho}(\vec{y}, \vec{z})}{\prod_{k=1}^m\abs{B_{{\rho}}(y_k,{\rho}(\vec{y},\vec{z}))\cap B}}, \quad {\rho}(\vec{y}, \vec{z})\le C_1 \,{\rho}(\vec{x},\vec{y}).
\end{eqnarray*}
Consider the Carnot-Carath\'eodory space given by
$\Omega^{m}$ and the H\"ormander vector field $\tilde{Y}$ made of $m$ copies of $Y$ and let $\tilde{\rho}$ be the associated Carnot-Carath\'eodory metric.
 Recalling that $(B,\rho, \text{Lebesgue measure})$ is a space of homogeneous type with uniform constants, using  Lemma~\ref{productdistance} and \eqref{Bdoubling},  the growth condition stated above reduces to have that for every positive constant $C_1$
there exists a positive constant $C_2=C_2(\Omega,\Omega_0,Y)$ such that for all $\vec{x},\,\vec{y},\,\vec{z}\in B^m,$
\begin{eqnarray*}
\frac{\tilde{\rho}(\vec{x}, \vec{y})}{\abs{B_{\tilde{\rho}}(\vec{x},\tilde{\rho}(\vec{x},\vec{y}))}}&\le& C_2\, \frac{\tilde{\rho}(\vec{z}, \vec{y})}{\abs{B_{\tilde{\rho}}(\vec{z}, \tilde{\rho}(\vec{z},\vec{y}))}}, \quad\text{ if } \tilde{\rho}(\vec{z}, \vec{y})\le C_1 \,\tilde{\rho}(\vec{x},\vec{y})\\
\frac{\tilde{\rho}(\vec{x}, \vec{y})}{\abs{B_{\tilde{\rho}}(\vec{x}, \tilde{\rho}(\vec{x},\vec{y}))}}&\le& C_2\, \frac{\tilde{\rho}(\vec{y}, \vec{z})}{\abs{B_{\tilde{\rho}}(\vec{y}, \tilde{\rho}(\vec{y},\vec{z}))}}, \quad\text{ if } \tilde{\rho}(\vec{y}, \vec{z})\le C_1 \,\tilde{\rho}(\vec{x},\vec{y}).
\end{eqnarray*}
These inequalities follow from the reverse doubling property in the Carnot-Carath\'edory space given by $\Omega^m$ and $\tilde{Y}$ (see Remark \ref{reversedoubling2}) and the doubling property of Lebesgue measure on $\rho$-balls with center in $\Omega_0$.

Finally, the weight conditions \eqref{2wq>1hm} and \eqref{2wq<1hm} with $\alpha=1$ for balls in $(B,\rho)$ and with constants that do not depend on $B,$  follow from \eqref{2wq>1} and \eqref{2wq<1}, respectively, and \eqref{Bdoubling}.
\end{proof}

\end{document}